\documentclass[12pt, journal,onecolumn]{IEEEtran}
\usepackage[ruled]{algorithm2e}
\usepackage{fullpage}
\usepackage{cite}
\usepackage{amsmath, amsthm, amssymb,graphicx,bbm}
\usepackage{verbatim}
\usepackage{psfrag}
\usepackage{algorithm2e}
\usepackage{bm}
\usepackage{color}
\usepackage[tight,footnotesize]{subfigure}

\newcommand{\R}{\mathbb{R}}

\newcommand{\be}{\beta}
\newcommand{\al}{\alpha}

\newcommand{\bd}{\mathbf}

\newcommand{\tl}{\tilde}

\newcommand{\x}{\bd{x}}
\newcommand{\y}{\bd{y}}

\newcommand{\ov}{\overline}

\newcommand{\la}{\lambda}
\newcommand{\z}{\bd{z}}

\newcommand{\bma}{\begin{bmatrix}}
\newcommand{\ebma}{\end{bmatrix}}

\newcommand{\Del}{\Delta}
\newcommand{\nn}{\nonumber}

\newcommand{\g}{\mathbf{g}}

\newcommand{\mc}{\mathcal}

\newcommand{\hd}{\mathbf{\hat{d}}}
\newcommand{\bD}{\bm{\Del}}
\newcommand{\rld}{\mbox{rld}}

\DeclareMathOperator{\E}{\mathbb{E}}

\newtheorem{defn}{Definition}
\newtheorem{thm}{Theorem}
\newtheorem{claim}[thm]{Claim}
\newtheorem{lem}[thm]{Lemma}

\title{{ Network Risk Limiting Dispatch}:\\ Optimal Control and Price of Uncertainty}
\author{
Baosen Zhang,~\IEEEmembership{Student Member,~IEEE,}
Ram Rajagopal,~\IEEEmembership{Member,~IEEE,}
        David Tse,~\IEEEmembership{Fellow,~IEEE}
\thanks{All authors are with Stanford University. B. Zhang is with the Departments of Civil and Environmental Engineering and Management Sciences and Engineering. R. Rajagopal is with the Department of Civil and Environmental Engineering.  D. Tse is with the Department of Electrical Engineering. E-mails:\{zhangbao,ramr,dntse\}@stanford.edu}
}
\begin{document}
\maketitle

\begin{abstract}
Increased uncertainty due to high penetration of renewables imposes significant costs to the system operators.  The added costs depend on several factors including market design, performance of renewable generation forecasting and the specific dispatch procedure.  Quantifying these costs has been limited to small sample Monte Carlo approaches applied specific dispatch algorithms. The computational complexity and accuracy of these approaches has limited the understanding of tradeoffs between different factors. {In this work we consider a two-stage stochastic economic dispatch problem. Our goal is to provide an analytical quantification and an intuitive understanding of the effects of uncertainties and network congestion on the dispatch procedure and the optimal cost.} We first consider an uncongested network and calculate the risk limiting dispatch. In addition, we derive the price of uncertainty, a number that characterizes the intrinsic impact of uncertainty on the integration cost of renewables. Then we extend the results to a network where one link can become congested. Under mild conditions, we calculate price of uncertainty even in this case. We show that risk limiting dispatch is given by a set of deterministic equilibrium equations.  The dispatch solution yields an important insight: congested links do not create isolated nodes, even in a two-node network.  In fact, the network can support backflows in congested links, that are useful to reduce the uncertainty by averaging supply across the network.  We demonstrate the performance of our approach in standard IEEE benchmark networks.

\end{abstract}
\section{Introduction}

The existing electric grid is operated so that online generation is sufficient to meet peak period demand. But \emph{uncertainties} arising from outages and unpredicted fluctuations in demand and renewable generation can cause a loss of load event, when online generation does not meet demand some load needs to be disconnected from the power system.  To decrease the loss of load probability, the system operator (SO) schedules generation and transmission line capacity so it exceeds forecasted peak net demand by a small percentage (around 5\%), to compensate for small amount of uncertainty due to generator contingencies and load forecast errors. This additional reserve capacity is utilized in real time as actual loads and contingencies are revealed.  Typically energy and reserve capacity are scheduled following a `3-$\sigma$' rule: the total amount scheduled is the forecast plus `3-$\sigma$', where $\sigma$ is the standard deviation of net demand forecasting error. Currently, the typical values of $\sigma$ is around 1\% to 2\% of total load.

Due to various incentives and state goals such as the renewable portfolio standards (RPS), renewables are expected to make up to 30\% to 40\% of generation mix in the USA. Increased penetration of renewable generation increases the uncertainty in the grid \cite{MilliganKirby:09, Makarov:09}. In such scenario, the current deterministic dispatch practice would require large reserve capacity allocations.  Such allocations increase energy costs significantly and accrue unwanted emissions \cite{RBDEWV2012}. {For example, each additional $1\%$ of reserve costs CAISO about $50$ million dollars (based on 2009 costs). In light of the significant financial implications, various alternative forms of stochastic dispatch procedures have been studied \cite{Morales:2009, Xie:10}.} The goal of these procedures is to solve a dispatch program that utilizes available forecasts and the sequential decision nature of the problem.  Past approaches often resulted in programs that were infeasible in practice due to computational complexity and relying { on} Monte-Carlo type approaches that could only be calculated with a limited number of scenario samples. The complexity of these procedures makes it even difficult to reliably evaluate the benefit of smart grid technologies or improvements in forecasting.  Moreover, these approaches require significant changes in the operating procedures and software of system operators.   In some cases, the forecast error distributions are not utilized appropriately or at all \cite{DeCesaroetal:09}.

Recently, Risk Limiting Dispatch (RLD) \cite{Varaiya11, RBVW2011} was proposed as { a} new dispatch framework.  By utilizing a simplified approach that is applied after unit commitment and does not consider network constraints, a very simple analytic dispatch rule can be obtained. The rule proposes an alternative deviation calculation that depends on error performance of forecasting, the costs of various generation alternatives and the timing of dispatch decisions stages.  It was shown that in uncongested and lossless networks, the proposed dispatch significantly reduces the renewable integration cost.  Moreover, reliable estimates of various metrics such as integration cost, emissions and costs due to forecasting performance can be easily obtained \cite{RBDEWV2012}.  

The first contribution of this paper is the derivation of risk limiting dispatch for a \emph{congested network}. This dispatch is denoted the network RLD and we show that it is simple to implement computationally(without the need for Monte Carlo type of simulations),  results in reliable and interpretable dispatch decisions and can be used to provide stable performance estimation. We model economic dispatch under uncertainty as a two-stage dispatch problem where the decision is made for each operating hour. Without loss of generality, we assume that the first stage occurs at the day ahead market and the second stage occurs at the real time market. In a day-ahead market (DAM), the SO purchases energy at generators connected to different buses in the network, utilizing forecasts and error distributions for loads and renewable generation at various buses. In the real-time market (RTM), dispatch decisions are made utilizing the realized values of all loads, renewable generations and physical network constraints such as transmission limits. We consider a DC power flow model for analysis and validate our results by considering full AC model in case studies.   

The key observation that makes the problem tractable is that in real networks, only a \emph{very small number of transmission lines are congested}. For example, the commonly used IEEE benchmark networks \cite{IEEEbenchmark} are far from being congested under normal operations. Also, the WECC model for the California network only include a few congested lines \cite{Price11}. We expect that the congestion patterns would not shift excessively under the uncertainty levels typically present in the renewable penetration levels expected in the near future. Intuitively, knowing the congestion patterns should reduce the complexity of the dispatch procedure since not all possible network constraints need to be considered. In this paper we formalize this intuition by developing an accurate picture of a network operating under \emph{expected congestion}, that is where congestion is predicted in the DAM. We observe a novel fact: a network operating under expected congestion due to uncertainty behaves qualitatively different than a network under deterministic loads and generation. We introduce the concept of \emph{back flow} to capture this behavior.  Back flows are directed permissible flows in congested links that need to be included in a two stage dispatch.
The possibility of back flow is somewhat surprising, as congestion in a two bus network in deterministic dispatch program implies the two buses are decoupled  \cite{SuGamal2012, Kirschen04}. We also develop a computationally simple dispatch approach that utilizes this structural understanding to compute the dispatch in a simplified form via a set of equilibrium equations.  The proposed approach can be easily integrated into existing unit commitment and dispatch procedures of the system operators. 

The second contribution of this paper is in developing the concept of \emph{price of uncertainty} that characterize the intrinsic impact of uncertainty on the cost of dispatch. Given a network, the integration cost is defined as the difference between expected cost under the \emph{optimal} dispatch procedure (i.e., RLD)  and the dispatch cost if the SO has a clairvoyant view of all loads and renewable generations  \cite{Rajagopal2012PMAPS}. We observe that under the expected mild to moderate uncertainty levels, the integration cost increases \emph{linearly} with the uncertainty in the forecast error and the per unit of uncertainty cost of integration is the price of uncertainty. 
The price of uncertainty can also be interpreted as the benefit of improving forecast procedures and can be used as a metric to evaluate the benefits of forecasting and provide a reference point to judge specific dispatch methodologies.


{ A brief discussion of related works follow.  Monte Carlo based dispatch formulations that include security constraints and DC power flow balance have been studied recently  \cite{Bouffard:2008, Morales:2009, Saric2009,Papavasiliou11,Tuohy09,Ruiz09,GK03}. They result in difficult optimization problems that can only be evaluated with (limited) Monte Carlo runs and do not provide much insight into the dispatch methods. MPC approaches \cite{GT2011, IXJ2011} address recourse in decision making, but still rely on Monte Carlo, and may not be appropriate when the number of recourse opportunities is small, limiting the corrections calculated by MPC. Single market problems are more tractable \cite{Morales:10, Pinsonetal:07, MatevosyanSoder:06, Bathurst:02} but do not capture the nature of recourse or congestion.  Methodologies for assessing reserves in the presence of significant wind generation was presented in \cite{DohertyOMalley:05} without including two stages or congestion. Current deterministic dispatch avoids complicated procedures by considering a worst-case net load to be satisfied, namely the forecast plus three standard deviations of forecast error \cite{boyle2008renewable}.  Other papers (e.g. \cite{Chen2012, Street2011}) investigated a robust version of unit commitment utilizing a DC flow model without recourse to represent the market model, and \cite{Bertsimas11} used a similar model but is fully adaptive to the realization of the uncertainties.
}

The remainder of the  paper is organized as follows.  Section \ref{sec:model} sets up the two-stage dispatch model in detail, and describes the uncertainty model.  Section \ref{sec:single} reviews a single bus model and develops the price of uncertainty. In order to develop this qualitative understanding under limited congestion patterns we first study small network scenarios. Section \ref{sec:Networks} first investigates a 2-bus network, defines the concept of back flow and identifies the appropriate structural results, utilizing it to develop a simple dispatch methodology. Section \ref{sec:Networks} then investigates general networks with a single congested link and demonstrates an appropriate reduction mechanism.  Section \ref{sec:simulations} provides computational experiments illustrating the performance of the procedure in real networks.  Section \ref{sec:con} concludes with future work.



\section{Model Setup} \label{sec:model}
\subsection{{Network Risk Limiting Dispatch (N-RLD)}}\label{subset:two-stage}

 \begin{figure}[ht]
\centering
\psfrag{g}{$g$}
\psfrag{gr}{$g^r$} 
\psfrag{D1}{$d_1$}\includegraphics[scale=0.4]{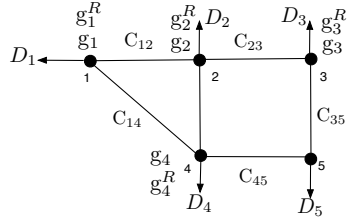}
\caption{Power network example with five nodes. Demand at bus $i$ is denoted by $D_i$, the first stage generation by $g_i$ and second stage generation by $g_i^R$. The capacity of line between bus $i$ and bus $k$ is denoted $C_{ik}$ and the flows on each line is determined by the net injection, $g_i+g_i^R-D_i$ at each bus. }%
\label{fig:dispatch}%
\end{figure}

Network Risk Limiting Dispatch (N-RLD) is formulated as a two stage optimization problem  in an power network (Figure~\ref{fig:dispatch}). The first stage represents a market where the SO can buy energy corresponding to dispatch decisions. Decisions are made at each node of the network.   The second stage corresponds to delivery or real time, which represents a 5 or 15 minute interval during which energy is delivered. Stage $1$  typically occurs 24 hours ahead of real-time and slow ramping generation or base load generation is dispatched at this stage \cite{Kirschen04}. In some cases, stage $1$ can represent a market an hour ahead of real-time. Without loss of generality we call this stage `day ahead'.  Stage 2 is then the `real-time'.

The SO makes dispatch decisions constrained by an $n$-bus power network  with $m$ transmission lines. The SO has to satisfy a random load $l_i$ at each bus $i$, known only at real-time. He has available for free $w_i$ units of renewable energy, also revealed only at real-time. In stage $1$, the SO can dispatch generation $g_i$ after observing some information \footnote{For example, the information observed in the day-ahead could be the weather information.} available about the random load and wind at the buses. In stage $2$ or recourse stage, the SO chooses  $g_i^R$ after observing the random loads and demands to balance the network. { The $g_i^R$'s can be seen as generation level of fast generators or shedded load.} Energy can only be purchased in the first stage so $g_i \geq 0$.   Renewable generation is not dispatchable and is taken as {\it negative load}, following standard practice. The net load at bus $i$ is defined as $d_i = l_i - w_i$ and it can be positive or negative. Excess power at any bus $i$ in the second stage can be disposed off for free, so  $g_i^R$ can be positive or negative.

The cost of dispatching generation at bus $i$ in the first stage is $c_i(g_i)$ and $q_i(g_i^R)$ in the second stage. In general both costs are represented by increasing, convex functions. When not specified, we assume that both costs are linear and given by $c_i(g_i) = \alpha_i g_i$ and $q_i(g_i^R) = \beta_i (g_i^R)^+$, where $\alpha_i$ and $\beta_i$ are prices in dollars per MW and $(x)^+ = \max(x,0)$. Later we show the assumption is not restrictive since while operating under mild to moderate uncertainty, we are interested in perturbations of the dispatch around its operating point, and it can be shown that the linear cost segment at that point determine costs.  Moreover, to avoid trivial solutions and arbitrage, assume day ahead prices are smaller than real time prices, i.e.,  $\al_i \leq \be_k$ for all $1 \leq i,k \leq n$. The total SO cost for the first stage is  the sum of first stage costs, and the total second stage cost is the sum of the second stage costs.

Once first stage dispatch decisions are made and renewable generation is realized, second stage dispatch decisions and power flows in the system are determined by the physical network and its properties. We consider a DC power flow model~\cite{Stott09} for dispatch calculation purposes.   We validate the performance of the dispatch by considering the full AC model in the case studies. Dispatch decisions need to respect network flow constraints, and in particular transmission line constraints. The capacity of the line connecting nodes $i$ and $j$ in the network is given by $c_{ij}$.  We also utilize an observation about congested transmission links in practice. For example, in CAISO, normally only one or two of the main transmission lines from Northern California to Southern California experience congestion. This {\it limited congestion} assumption will be utilized to simplify the dispatch calculation. In particular in this paper we focus on the scenario with at most one congested link. In future work we generalize this to problems with $k$ congestion link following the same approach proposed here.

{ 
To establish the information structure of the two stage optimization problem, we propose the following forecasting model. The net load is decomposed as
\begin{equation}
\bd{{d}} = \bd{\hat{d}}+\bd{e},
\end{equation}
where $\bd{\hat{d}}=[\hat d_1 \hat d_2 \dots \hat d_n]^T$ is the first stage forecast  and $\bd{e}=[e_1 e_2 \dots e_n]^T$ is a zero mean Gaussian distributed random vector with covariance matrix $\Sigma_e = \sigma_e^2\Sigma'_e$, where $\Sigma_e'$ is a known error correlation matrix ($\bd{e} \sim N(\bd{0},\sigma_e^2\Sigma'_e)$). Furthermore, the forecast $\bd{\hat{d}}$ and error  $\bd{e}$ are independent.  Moreover, average performance of the forecast, in the form of forecast error variance $\sigma^2_e$ is provided for each bus or operating region.

The Gaussian error assumption is justified by recent studies (e.g.~\cite{Makarov:09}) that observe forecast errors are distributed as a (truncated) Gaussian  random variable. For the typical variances utilized, the errors in utilizing a Gaussian distribution is negligible and an accepted practice in dispatch mechanisms. Also note that the results for Gaussian models carry over to many other distributions with little modification \cite{RBVW2011}.

}
{
\subsection{Formulation and Decomposition of N-RLD}

We formulate the mathematical optimization problem in this section. Before stating the entire problem, it is convenient to define the following DC-OPF problem
\begin{subequations}\label{eqn:RT_X}
\begin{align}
 J(\bd{q},\bd{x}) =  & \min_{\g^R,\bd f} \bm{q}^T (\g^R)^+ \label{eqn:cost_X}\\
 &\mbox{subject to } \g^R-\x - \bm{\nabla}^T \bd{f} =0 \label{eqn:nabla},\\
&  \hspace{0.75in} \bm{K} \bd{f} =0 \label{eqn:Kirchoff},\\
 & \hspace{0.75in}  |\bd{f}| \leq \bd{c} \label{eqn:capacity},
\end{align}
\end{subequations}
where $\bd q$ is a positive price vector,  \eqref{eqn:nabla} is the power balance constraint, $\bm{\nabla}^T \in \R^{n \times m}$ is the mapping from branch flows to bus injections \cite{Kanoria11}, $\bd{f}$ is the  $m \times 1$ vector of branch flows, \eqref{eqn:Kirchoff} is Kirchoff's voltage law that states a weighted sum of flows in a cycle must be $0$, \eqref{eqn:capacity} are the capacity constraints on the flows where $\bd{c}=[c_1 \; \dots \; c_m]^T$ and $(x)^+=\max(0,x)$. This optimization problem can be seen as a generic DC-OPF problem with prices $\bd{q}$ and demands $\bd{x}$. Since only the positive part of generations $\g^R$ is reflected in the cost, energy can be disposed for free.   


\noindent \textbf{N-RLD}: The network  risk limiting dispatch problem can be stated as the following stochastic optimization problem:
\begin{itemize}
\item[(i)] Real Time OPF (RT-OPF):
Solve the real time OPF problem $J(\bm{\beta},\bd{d}-\bd{g})$ where $J$ is defined in \eqref{eqn:RT_X}. At real time, the day-ahead dispatch decisions $\bd{g}$ are already made, and the realization of the random variables are known. Therefore the new net demand is $\bd{d}-\bd{g}$, and $J(\bm{\beta},\bd{d}-\bd{g})$ balances the network under the real time prices $\bm{\beta}$.   

\item[(ii)] Day Ahead Stochastic Power Flow (DA-SPF):
\begin{align}\label{eqn:DA-SPF}
V^*(\hat{\bd{d}})&=\min_{\g \geq 0}\left\{\bm{\al}^T \g + \E[J(\bm{\be},\bd{d}-\g)|\hat{\bd{d}}] \right\},
\end{align}
where the expectation is taken with respect to the distribution of $\bd{d}$ conditional on the forecast $\hat{\bd{d}}$.  The constraint $\g \geq 0$ limits the day ahead decisions to purchasing generation power only. Additionally, $\g $ is function of the forecast $\bd{\hat{d}}$ and the error distribution. The optimal solution to \eqref{eqn:DA-SPF} is called the \emph{risk limiting dispatch}.    
\end{itemize}
}
{
\subsection{Integration Cost and Price of Uncertainty}

A fundamental quantity of interest is the impact of uncertainty in the cost of dispatch. We call this quantity the integration cost  \cite{RBDEWV2012}, which is defined   the  {\it difference between the expected cost of the procedure and the expected cost of a dispatch clairvoyant of the realization of $\bd{d}$}.
The clairvoyant dispatch can allocate all the required power in the day ahead by solving the deterministic OPF $V_{C}^*(\bd{d}) = J(\alpha,\bd{d})$. 
The integration cost for a realization of  the information set $\hd$ is given by
\begin{align}\label{eqn:cost}
C_I(\hd) = V^*(\hd) - \E[V_{C}^*(\hd + \bd{e})|\hd].
\end{align}
An important question is regarding the sensitivity of this cost to the forecast error standard deviation $\sigma_e$ when the {\it best possible dispatch} is utilized.  If $C_I$ is a linear function of $\sigma_e$, so $C_I = p \sigma_e$, then $p$ is the {\it price of uncertainty}, a fundamental limit faced by {\it any} dispatch procedure. In this paper we show how it can be calculated for various scenarios.

\subsection{Small-$\sigma$ Assumption}

An important consideration  is the order of magnitude of the error standard deviation $\sigma_e$ compared to the entries in the average net load vector $\bd{\mu}$ and the transmission line capacities. Standard deviation of  day ahead load forecasts $\sigma_L$ are $1\%--2\%$ of the expected load $\mu_L$. Wind error forecasts are more severe, and error standard deviations $\sigma_W$  of $30\%$ of rated capacity $\mu_W$ have been observed. High wind penetration scenarios have about $30\%$ of total load being generated by wind, and therefore the total error would be about $0.01+0.3*0.3=10\%$ of total load.

In contrast to the financial situation, a relative forecast error of $10\%$ would not change the overall physical operating characteristic of the network. More precisely, suppose we calculate the deterministic dispatch based on the forecast values $\hat{\bd{d}}$ and find bus $i$ would be generating power in the first stage. Then with high probability, bus $i$ would still be generating power in the two stage dispatch problem. Also, the network congestion pattern under the deterministic dispatch and the two-stage dispatch should not be drastically different. Sections \ref{sec:single} and \ref{sec:Networks} formalizes these observations. 

We call the operating regime in the above scenarios the {\it small}-$\sigma$ regime. More rigorously, we have the following definition. 
\begin{defn} \label{defn:small}
Let $\hd$ be the predicted net demand and $\sigma_e$ be the standard deviation of the forecast error. The small-$\sigma$ assumption denotes the scaling regime where $ \frac{1}{\sigma_e} \hd \rightarrow \infty $.
\end{defn}
  For the simplicity of exposition, we delegate such limits to the appendix and focus on the intuitive  points of analysis in the main body of the paper. The overall message is that forecast values are very useful in determining the \emph{qualitative} behaviour of the network.
}

\section{{ Single-Bus Network Case}} \label{sec:single}

This section reviews the risk limiting dispatch control for a single-bus network
~\cite{RBVW2011, RBV2011}, and analyzes the price of uncertainty in this scenario. A network can be modeled by a single-bus if congestion never occurs in the network. Under the same-$\sigma$ assumption, this is equivalent to the fact that if there is sufficient capacity under the forecast net-demand, then the forecast errors being small enough compare to the capacity in the network such that line flow limits would not be hit under almost all realizations. 

\subsection{Risk Limiting Dispatch}
Since we only consider a single bus, all variables are scalar. Equivalently, the single bus network can be thought as an $n$-bus network without congestion since buses can freely exchange power\footnote{More precisely, this fact follows from the fact that without congestions, Kirchoff's laws reduces to the law of conservation of energy, which only requires the total power input to be equal to the total power output.}.  In this case, the constraint region in \eqref{eqn:RT_X} reduces to net supply must equal net demand, and the RT-OPF becomes
\begin{align*}
J^*(\be,d-g) = & \min \mbox{ } \be (g^R)^+ \\
&\mbox{s.t. } g^R+g-d =0 \\
=& \beta (d-g)^+.
\end{align*}
The DA-SPF in \eqref{eqn:DA-SPF} can then be reduced to
\begin{subequations} \label{eqn:1bus_opt}
\begin{align}
g^*= \arg \min_g & \al g + \be \E[(d-g)^+|\hat{d}]\\
 \mbox{s.t. } & g \geq 0.
\end{align}
\end{subequations}

RLD can then be derived as follows. Consider the unconstrained optimization problem
\begin{equation} \label{eqn:uncon}
\min_g \; \al g + \be \E[(d-g)^+|\hat{d}].
\end{equation}
Taking the subgradient with respect to $g$ gives the optimality condition
\begin{align*}
0 &= \al - \be\E[\bd{1}(d-g >0)|\hat{d}] \\
  &= \al - \be\E[\bd{1}(\hat{d}+e-g>0)|\hat{d}] \\
  &= \al - \be \Pr(e>g-\hat{d}|\hat{d}),
\end{align*}
rearranging gives
\begin{equation} \label{eqn:phi}
\Pr(e>g-\hat{d}|\hat{d})=Q(g-\hat{d})=\frac{\al}{\be},
\end{equation}
where $Q(\cdot)$ is the Gaussian Q function. The risk limiting dispatch (optimal dispatch) $g$ is given by inverting \eqref{eqn:phi}
\begin{equation}\label{eqn:rld}
g=\hat{d}+Q^{-1}(\frac{\al}{\be}).
\end{equation}
Note it is possible that $g<0$, it can be shown that the risk limiting dispatch $g^*$ (optimal solution to the constrained problem in \eqref{eqn:1bus_opt}) is given by thresholding
\begin{equation}
g^*=g^+=[\hat{d}+Q^{-1}(\frac{\al}{\be})]^+.
\end{equation}

\subsection{Price of Uncertainty}
Since most power systems would not have $100\%$ penetration in the near future, we assume that the net demand $d$, and its prediction $\hat{d}$, are positive.  Then first we would show the price of uncertainty exists (i.e. the integration cost is linear in $\sigma_e$), and then calculate its value.
\begin{thm} \label{thm:price_single}
Suppose $d>0$. Then $C(\hat{d})$ defined in \eqref{eqn:cost} is linear under the small-$\sigma$ assumption and can be written as
\begin{equation} \label{eqn:cost_linear}
\lim_{\frac{1}{\sigma_e} \hat{d} \rightarrow 0} C(\hat{d})= \sigma_e p,
\end{equation}
where $\sigma_e$ is the standard deviation of the error $e$ and { $p=\beta \phi (Q^{-1} (\frac{\alpha}{\beta}))$ ($\phi(\cdot)$ is the standard Gaussian density and $Q(\cdot)$ is the complimentary Gaussian cumulative density function).} 
\end{thm}

Theorem~\ref{thm:price_single} relies on the observation that if net demand is positive ($d>0$), then it is always beneficial to purchase energy in the day ahead as the energy price is higher in real-time, so the optimal schedule must be positive $g^* > 0$. The positivity constraint in the simplified DA-SPF (\eqref{eqn:1bus_opt}) is redundant, and the cost of uncertainty (\eqref{eqn:cost}) can be explicitly computed. The proof of Theorem \ref{thm:price_single} is given in Appendix \ref{app:price_single}. 
Figure \ref{fig:p_single} plots the price of uncertainty for different values of $\al/\be$ with $\be$ set to be 1. Somewhat surprisingly, $p$ is not monotonic in $\al/\be$ and it goes to $0$ as $\al/\be$ approaches $0$ or $\al/\be$ approaches 1. Intuitively, when $\al/\be$ is small, the day ahead cost is very low, and the SO can purchase sufficient amounts of energy to absorb the prediction error. In contrast, when $\al/\be$ is close to 1, the day ahead and real-time costs are similar, so the SO waits until real-time to balance the system once the net load realization is completely known.
\begin{figure}[ht]
\centering
\includegraphics[scale=0.35]{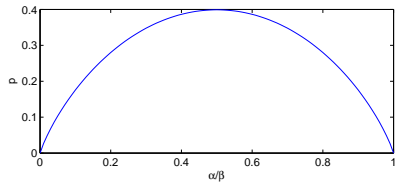}%
\caption{The price of uncertainty for different ratios of $\al/\be$.}%
\label{fig:p_single}%
\vspace{-0.5cm}
\end{figure}
\subsection{Extremely High Penetration}

In some networks renewable power may have a penetration level of more than 100\%, violating the small-$\sigma$ assumption. For example, in a microgrid where wind or solar energy is abundant, the net demand could become negative.
In this case, the cost of uncertainty is no longer linear in the standard deviation of the prediction error and in general cannot be computed in closed form. 

\section{Congested Networks Case} \label{sec:Networks}

The RT-OPF in N-RLD for n-bus networks does not admit an analytical solution as in the single bus case, significantly increasing the complexity of the full dispatch. In particular, it is difficult to obtain the day ahead dispatch $\bd{g}$ in closed form. Moreover, the cost of uncertainty can be a complicated function of the information set and the network capacities $\bd{c}$. These quantities can be numerically computed resorting to a Monte Carlo approach, but the computational challenges are formidable due {to the} high dimensionality of the problem.

{ Instead, the small-$\sigma$ assumption from Sec.~\ref{sec:model} can be explored to obtain a simple and interpretable dispatch. Since the prediction error is a small percentage of the net load, the change in flows caused by that error is also a small percentage, we assume the prediction error is small compared to both $\hd$ and $\bd{c}$.} 
Under the small-$\sigma$ assumption, the {\it qualitative} or structural behavior of the power system predicted in the day-ahead from the forecast $\hd$ will not differ from its realization in real-time after observing $\bd{d}$. If we expect to purchase power at a bus in the day-ahead, then after real-time, we do not expect power to be shed in that bus.  If a transmission line is expected to be congested in a certain direction in the day-ahead, then the direction of congestion would not be reversed at real time. Since qualitative features are consistent with the forecast, a deterministic OPF based on the day-ahead price $\bm{\al}$ and the net load forecast $\hd$ will predict congested lines, congestion directions and buses where energy is purchased correctly. Let $\mathcal P$ denote the feasible injection region of the network (the set of all power injections that satisfy the operational constraints) \cite{Zhang13}.  {This deterministic OPF is denominated {\it Nominal Day-Ahead OPF} (NDA-OPF):} 
\begin{subequations}\label{eq:NDA-OPF}
\begin{align}
J(\bm{\al},\hd) = \min_{\g} &\bm{\al}^T (\g)^+ \\
 \mbox{subject to } & \g-\hd  \in \mc{P}.
\end{align}
\end{subequations}
In stochastic control terms, NDA-OPF solves the certainty equivalent control problem for N-RLD \eqref{eqn:DA-SPF}~\cite{Kurzhanski01,Theil57}, by replacing the random quantity $\bd{d}$ by the deterministic quantity $\hd$ and solving the optimization problem. Denote the generation schedule from NDA-OPF by $\ov{\g}$. 

The day ahead schedule $\g$ in the DA-SPF (\eqref{eqn:DA-SPF}) can be decomposed as the nominal dispatch added to a perturbation $\g=(\ov{g}+\bD)^+$ where $\bD \in \R^n$ is the perturbation. The optimal schedule is determined by computing $\bD$. Perturbations are expected to be small since the uncertainty is small, so the {\it perturbed DA-SPF} can be significantly simplified. The simplification comes from the small-$\sigma$ assumption (see Definition \ref{defn:small}), and is manifested in three key observations:
\begin{enumerate}
	\item If $\ov{g}_i <0$, bus $i$ is treated as a source of unlimited energy, since it is shedding energy in the nominal problem (NDA-OPF). If $\ov{g}_i >0$, then the perturbed dispatch $\Del_i$ is not constrained to be positive since $ \overline{g}_i$ is much larger than $\sigma_e \Delta_i$ under the small-$\sigma$ assumption. 
	\item If the line between buses $i$ and $k$ are not congested, then it is not congested in perturbed DA-SPF.
	\item If the line between buses $i$ and $k$ is congested from $i$ to $k$, then it would not become congested from $k$ to $i$ in the perturbed DA-SPF.
\end{enumerate}
Going forward, we assume these observations to hold. This is called the \textbf{small-$\bm{\sigma}$} assumption. We propose the two step algorithm in Algorithm 1. 

{
\noindent \textbf{Algorithm 1:} Procedure to solve Network Risk Limiting Dispatch\\
\textbf{Step 1 (NDA-OPF):} Solve the nominal problem in \eqref{eq:NDA-OPF} using forecast net load and day ahead prices to obtain the nominal schedule $\ov{\g}$ and nominal line flows  $\ov{\bd{f}}$. \\
\textbf{Step 2 (Perturbed DA-SPF):}Solve the DA-SPF (\eqref{eqn:DA-SPF}) for the optimal perturbation $\bD$ by substituting $\g = \ov{\g} + \sigma_e \bD^*$ and appropriately normalizing and reducing the problem using Observations $(1)-(3)$ as
\begin{subequations} \label{eqn:perturb_DASPF}
\begin{align}
\bD^* = \arg \min_{\bD} & \bm{\alpha}^T \bD + \E[\tilde{J}(\bm{\beta},\bd{e})|\hd] \\
\mbox{subject to } & \Delta_i =0 \mbox{ if } \ov{g}_i <0, \\
& \Delta_i > 0 \mbox{ if } \overline{g}_i =0,
\end{align}
\end{subequations}
where 
\begin{subequations} \label{eqn:perturb_OPF}
\begin{align}
\tl{J}(\bm{\beta},\bd{e}) = \min & \tl{\bm{\beta}}^T (\y)^+ \\
\mbox{s.t. } & \y - \bd{e} - \nabla^T \bd{f} = 0 \\
& \bd{K} \bd{f} =0 \\
& f_{ik} < 0 \mbox{ if } \overline{f}_{ik}=c_{ik}, \label{eqn:perb_flow}
\end{align}
\end{subequations}
and $\tl{\beta}_i = \beta_i$ if $\ov{g}_i \geq 0$ and $\tl{\beta}_i=0$ otherwise. The optimal DA-SPF dispatch is then given by $\bd{g}=(\ov{\bd{g}}+\sigma_e \bD)^+$. 

At first glance, \eqref{eqn:perturb_DASPF} seems to be no simpler than the original problem in \eqref{eqn:DA-SPF}. However, note that the network capacity constraints \eqref{eqn:perb_flow} only include the lines that are congested in the nominal problem. In essence, \eqref{eqn:perturb_DASPF} balances a 'left-over' network from solving the nominal problem, and $\eqref{eqn:perb_flow}$ states that if a line is congested in the nominal problem, no more energy is allowed to flow along the direction of congestion. }

The next subsection explores the normalization and reduction process to define the Perturbed DA-SPF for two bus and three bus networks. We show the perturbation $\bm{\Delta}$ is the solution to a set of deterministic equilibrium equations. Then the problem of an arbitrary network with  $n$ buses and a single congestion link is studied 
and we show the general reduction procedure results in an optimal dispatch control under the small-$\sigma$ assumption.

\subsection{Two Bus Network} \label{sec:two_bus}
Consider the two bus network in Fig. \ref{fig:2busD}.
\begin{figure}[ht]
\centering
\psfrag{D1}{$d_1$}
\psfrag{D2}{$d_2$}
\psfrag{C}{$c$}
\includegraphics[scale=0.85]{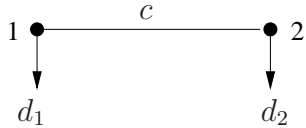}%
\caption{A two-bus network where $c$ is the capacity of the line.}%
\label{fig:2busD}%
\end{figure}
For this network, the day ahead dispatch is a vector $\g=[g_1\;g_2]^T$ of the scheduled generation at each bus. The real-time balancing of the network requires solving an OPF where the injection region is two dimensional.  The RT-OPF becomes
\begin{subequations}\label{eqn:2bus_J}
\begin{align}
J(\bm{\be},\bd{d}-\g) = \min_{\g^R,f} & \bm{\be}^T (\g^{R})^+ \\
												\mbox{subject to } & g_1^{R}+g_1 -d_1 -f =0 \\
												& g_2^{R}+g_2 - d_2+f =0 \\
												& |f|<c,
\end{align}
\end{subequations}
where $f$ is the amount of power flowing from bus 1 to bus 2 and $c$ is the capacity on the line.

To apply Algorithm 1, first solve the NDA-OPF (\ref{eq:NDA-OPF}) for the two bus network. Then, to apply Step 2, we partition $\R^2$ into the five regions in Fig.~\ref{fig:2_diff_al} according to the value of the net demand forecast $\hd$.  Each region is defined by whether the transmission link is congested or not, the direction of congestion, and whether each bus is scheduled to generate power in the nominal problem. The small-$\sigma$ assumption enables inference of these facts with high probability from the solution of the NDA-OPF.
\begin{figure}[ht]
\centering
\psfrag{D1}{$\hat{d}_1$}
\psfrag{D2}{$\hat{d}_2$}
\includegraphics[scale=0.45]{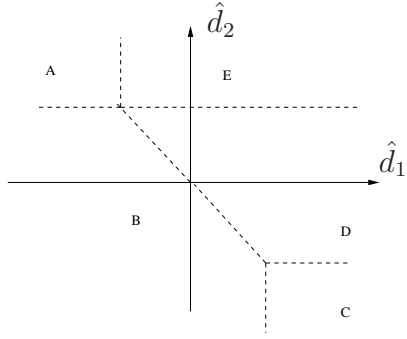}%
\caption{Partition of $\R^2$ with respect to $\hd$ when $\al_1 \leq \al_2$. The small-$\sigma$ assumption means that the actual realization of $\bd{d}$ is in the same region as $\hd$ w.h.p. }%
\label{fig:2_diff_al}%
\vspace{-0.5cm}
\end{figure}

Regions $A$, $B$, $C$ and $D$ reduces to the single bus case as analyzed in Section \ref{sec:single}. In regions $B$ and $D$, since the line capacity is not binding, power can flow from one bus to the other without congestion. In region $A$, bus 1 has excess power and transfer up to capacity to bus $2$, and then reserve is only needed for bus 2. Region $C$ is symmetrical to region $A$. 

For region E in Fig. \ref{fig:2_diff_al}, $\hat{d}_1> -c$ and $\hat{d}_2 > c$. Since buying at bus 1 is cheaper ($\al_1<\al_2$), the SO should transfer up to line capacity $c$ units of energy from bus 1 to bus 2. The NDA-OPF solution is then
\begin{equation*}
\ov{\g}= \bma \hat{d}_1+c \\ \hat{d}_2-c \ebma.
\end{equation*}
At first glance, it seems the two buses are now decoupled and can be treated as two isolated buses since the line between them is congested. However, this viewpoint is not correct due to the {\em two-stage} nature of the problem and congestion being directional. In the two stage dispatch problem, the SO decides in the first stage to purchase some energy based on the forecast and error statistics; however the actual balancing of the network occurs at the second stage. Some averaging of the errors can still occur even if the line from bus 1 to bus 2 is congested. For example, suppose that in real-time  $e_1>0$ and $e_2 <0$. That is,  demand at bus 2 was over-predicted and demand at bus 1 was under-predicted. Due to this configuration, bus 2 needs less than $c$ units of energy from bus 1, and the remaining energy can be utilized to satisfy the under-predicted demand in bus $1$. This represents a flow from bus $2$ to $1$ and does not violate congestion constraints, since the line was congested from bus $1$ to $2$. Due to this property of opposing the congestion direction, we denominate this flow a {\it backflow}. For example, backflow does not arise in region A because bus 1 always has an excess of energy and does not require any energy from bus 2. Similarly for region C.

In region $E$, the small-$\sigma$ assumption implies that $\bd{d} \in E$ with high probability and the line {is not congested} from bus $2$ to bus $1$ (Observation $(3)$).  Assuming that errors $e_1$ and $e_2$ have covariance matrix  
\begin{equation}
\Sigma_e= \sigma_e^2 \Sigma' = \sigma_e^2 \bma \gamma_{11} & \rho \\ \rho & \gamma_{22} \ebma, 
\end{equation}
the optimal dispatch and price of uncertainty in region E are given by:
\begin{thm} \label{thm:2_bus}
Consider the two-bus network in Fig. \ref{fig:2busD}, with prices $\al_1$ and $\al_2$ respectively. Without loss of generality, we assume $\al_1 \leq \al_2$. Under the small-$\sigma$ assumption, the risk limiting dispatch (equation \eqref{eqn:DA-SPF}) is given by
\begin{equation*}
\g^*= \ov{\g} + \sigma_e \bm{\Del^*},
\end{equation*}
where $\ov{\g}= [ \hat{d}_1+c \;\; \hat{d}_2-c]^T$ and $\bm{\Del^*}$ is the unique solution to
\begin{subequations} \label{eqn:equi_2bus}
\begin{align}
\al_1 &= \min(\be_1,\be_2) \Pr( z_1 > \Del_1,z_1+z_2 > \Del_1+\Del_2) \label{eqn:equi_1} \\
\al_2 &= \be_2 \Pr(z_2> \Del_2) \nn \\
	  & + \min(\be_1,\be_2) \Pr(z_2<\Del_2, z_1+z_2 >\Del_1+\Del_2),  \label{eqn:equi_2}
\end{align}
\end{subequations}
where $\z=[z_1 \; z_2]^T=\bd{e}/\sigma_e$. The cost of uncertainty is linear and the price of uncertainty is given by
\begin{align} \label{eqn:2bus_p}
p= & \bm{\al}^T \bm{\Del^*} \\
&+ \min(\be_1,\be_2) \{\E[(z_1+z_2-\Del_1^*-\Del_2^*)^+1(z_2<\Del_2^*)] \nn \\ &+\E[(z_1-\Del_1^*)^+ 1(z_2>\Del_2^*)]\}  + \be_2 \E[(z_2-\Del_2^*)^+] . \nn
\end{align}
\end{thm}

Before formally proving Theorem~\ref{thm:2_bus}, we provide an intuitive explanation of the non-linear equations in \eqref{eqn:equi_2bus}.
After subtracting the nominal dispatch choice, the net demands (normalized by $\sigma_e$)  are  $z_1$ and $z_2$ respectively, and only backflow is allowed. The network reduces to a two bus network with a {\em unidirectional} link going from bus 2 to bus 1 (Fig. \ref{fig:2bus_uni}).
\begin{figure}[ht]
\centering
\psfrag{D1}{$z_1$}
\psfrag{D2}{$z_2$}
\includegraphics[scale=0.85]{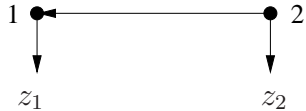}%
\caption{The perturbed network consisting of a unidirectional link and normalized demands $z_1=e_1/\sigma_e$, $z_2=e_2/\sigma_2$ .}%
\label{fig:2bus_uni}%
\end{figure}
 { The left hand side of \eqref{eqn:equi_2bus} can be seen as the cost of purchasing an additional unit of energy at the buses in stage 1, while the right hand side can be seen as the benefit of having that unit of energy at stage 2. Therefore \eqref{eqn:equi_2bus} can be interpreted as balancing the cost and benefit between buying an additional of unit at stage 1. For example, one additional unit of energy at bus 1 is useful if two event occurs: $z_1 >\Del_1$ (bus 1 does not have enough energy) and (b)  $z_1+z_2 > \Del_1 + \Del_2$ (bus 2 does not have enough energy to transfer to bus 1). Since power can be transferred from bus 2 to bus 1 in the perturbed network (Fig. \ref{fig:2bus_uni}), the price of buying an unit of energy at real time is $\min(\be_1,\beta_2)$ and the right hand side of \eqref{eqn:equi_1} is the expected benefit of having that unit of energy available. The price of purchasing that unit of energy at stage 1 is $\alpha_1$. At optimality, equilibrium is achieved between the cost at stage 1 and the expected benefit at stage 2. Similarly, \eqref{eqn:equi_2} describes the equilibrium at bus 2} 

%

Figure~\ref{fig:back_comp} plots the ratio in the average price between a network where backflow is not taken into account and a network that allows backflow as a function of the correlation between errors $e_1$ and $e_2$.
\begin{figure}[ht]
\centering
\includegraphics[scale=0.3]{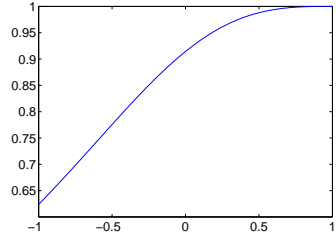}%
\caption{Ratio in prices between using and not using back flow for $\alpha_1=\alpha_2=0.5$ and $\beta_1=\beta_2=1$. Note the curve is always below one since a network with back flow can do no worse than a network without backflow. }%
\label{fig:back_comp}%
\end{figure}
If backflow is not allowed, then the network becomes two isolated buses. The ratio is always less than $1$ since a network with backflow can do no worse than a network without backflow.  The ratio is lowest when the two buses are negatively correlated since backflow averages out the uncertainties in the error. As the two buses become positively correlated, backflow becomes less useful since both errors tend to be the same sign and averaging is less useful.

\begin{proof}[Proof of Theorem \ref{thm:2_bus}] 
Note that Theorem \ref{thm:2_bus} can be proven using the same limiting arguments as given in Appendix \ref{app:price_single} for Theorem \ref{thm:price_single}. For the sake of clarity and brevity, we present a proof without going into the limiting details, but the arguments can be easily make rigorous by following Appendix \ref{app:price_single}. 
 
Any dispatch can be written as $ \ov{\g}+ \sigma_e \bD$. We first prove the optimal $\bD$ is independent of $\sigma_e$. Substituting $\g=\ov{\g}+\sigma_e \bD$, the DA-SPF (Eqn.~\eqref{eqn:DA-SPF})becomes
\begin{subequations} \label{eqn:2_bus_g}
\begin{align}
\mbox{minimize } & \bm{\al}^T (\ov{\g}+\sigma_e \bD)+ \E[J(\bm{\be},\bd{d}-(\ov{\g}+\sigma_e \bD))|\hd] \\
\mbox{subject to } & \ov{\g}+\sigma_e \bD \geq 0 \label{eqn:2bus_pos}.
\end{align}
\end{subequations}
By the small-$\sigma$ assumption, the constraint in Eqn.~\eqref{eqn:2bus_pos} is always satisfied since $\ov{\g} \geq 0$ from the definition of NDA-OPF. The RT-OPF  (Eqn.~\eqref{eqn:2bus_J}) becomes
\begin{subequations} \label{eqn:2bus_J_d}
\begin{align}
& J(\bm{\be},\bd{d}-(\ov{\g}+\sigma_e \bD)) \\
 = & \mbox{minimize } \bm{\be}^T (\g^{R+1})^+ \\
&\mbox{subject to }  g^{R+1}_1+\ov{g}_1+\sigma_e \Del_1-f-\hat{d}_1-e_1 =0 \\
&\hspace{1.5cm} g^{R+1}_2+\ov{g}_2 + \sigma_e \Del_2 +f -\hat{d}_2-e_2=0 \\
&\hspace{1.5cm}   -c \leq f \leq c.
\end{align}
\end{subequations}
Since the nominal flow is $c$, let $f=c-\delta$ with $\delta$ representing the backflow. Substituting the value of $\ov{\g}$ into Eqn.~\eqref{eqn:2bus_J_d}, 
\begin{subequations} \label{eqn:2bus_J_sig}
\begin{align}
& J(\bm{\be},\bd{d}-(\ov{\g}+\sigma_e \bD)) \\
 = & \mbox{minimize } \bm{\be}^T (\g^{R+1})^+ \\
& \mbox{subject to } g^{R+1}_1+ \sigma_e \Del_1 + \delta- e_1 =0 \\
& \vspace{1.5cm} g^{R+1}_2 + \sigma_e \Del_2 -\delta -e_2=0 \\
& \vspace{1.5cm} 0 \leq \delta \leq 2c.
\end{align}
\end{subequations}
By the assumption that the line does not congest from bus $2$ to bus $1$, the constraint $\delta<2c$ is always satisfied and can be dropped. Normalizing Eqn.~\eqref{eqn:2bus_J_sig} by $\sigma_e$ gives
\begin{subequations} \label{eqn:2bus_J_nom}
\begin{align}
& J(\bm{\be},\bd{d}-(\ov{\g}+\sigma_e \bD)) \\
= & \sigma_e \mbox{minimize } \bm{\be}^T (\g^{R+1})^+ \\
&\mbox{subject to } g^{R+1}_1+  \Del_1 + \delta- z_1 =0 \\
& \hspace{1.5cm} g^{R+1}_2 +  \Del_2 -\delta -z_2=0 \\
& \hspace{1.5cm} \delta \geq 0,
\end{align}
\end{subequations}
where the optimization variables $\g^{R+1}$ and $\delta$ have been normalized by $\sigma_e$ and $z_i:=e_i/\sigma_e$. Let $\tl{J}=J /\sigma_e$, and note that $\tl{J}$ only depdent of $\bm{\be}$ and $\bD$. Combining \eqref{eqn:2_bus_g} and \eqref{eqn:2bus_J_nom}, $\bD$ solves the unconstrained optimization problem
\begin{equation}
\min_{\bD} \; \al^T \bD + E[\tl{J}(\bm{\be},\bD)].
\end{equation}
To solve this optimization problem, we need the gradient of $E[\tl{J}(\bm{\be},\bD)]$ with respect to $\bD$. The optimization problem can be analytically solved to yield
\begin{align*}
\tl{J}(\bm{\be},\bD) &=
\begin{cases}
 \min(\be_1+\be_2) (z_1+z_2-\Del_1-\Del_2) \\
\;\mbox{ if } z_1+z_2 > \Del_1+\Del_2, z_2 < \Del_2 \\
 \min(\be_1+\be_2) (z_1-\Del_1)+\be_2 (z_2-\Del_2) \\
\;\mbox{ if } z_1 > \Del_1, z_2 > \Del_2 \\
\be_2 (z_2-\Del_2) \\
\; \mbox{ if } z_1 < \Del_1, z_2 > \Del_2\\
0 \mbox{ otherwise}
\end{cases} \\
&= \min(\be_1,\be_2)[(z_1+z_2-\Del_1-\Del_2)^+ 1(z_2<\Del_2) \\
&+(z_1-\Del_1)^+ 1(z_2>\Del_2)] + \be_2 (z_2-\Del_2)^+.
\end{align*}
Using the linearity of expectation and taking derivatives with respect to $\bD$ in $\al^T \bD + E[\tl{J}(\bm{\be},\bD)]$ gives \eqref{eqn:equi_2bus}.

Next we prove the price of uncertainty is given by Eqn.~\eqref{eqn:2bus_p}. The value of full knowledge optimization problem is $\E[J(\bm{\al},\bd{d})]$. The error is zero mean and by the small-$\sigma$ assumption, $\E[J(\bm{\al},\bd{d})]=\bm{\al}^T \ov{\g}$ where $\ov{g}$ is the nominal solution. The cost of uncertainty is
\begin{align*}
u &=\bm{\al}^T (\ov{g}+\sigma_e \bD)+ \E[J(\bm{\be},\bd{d}-(\ov{\g}+\sigma_e \bD))] - \E[J(\bm{\al},\bd{d})] \\
 & =\sigma_e (\bm{\al}^T \bD + \E[\tl{J}(\bm{\be},\bD)] \\
 & = \sigma_e p.
\end{align*}
\end{proof}
\subsection{N-bus Network with a Single Congested Line}
\begin{figure}[ht]
\centering
\includegraphics[scale=0.1]{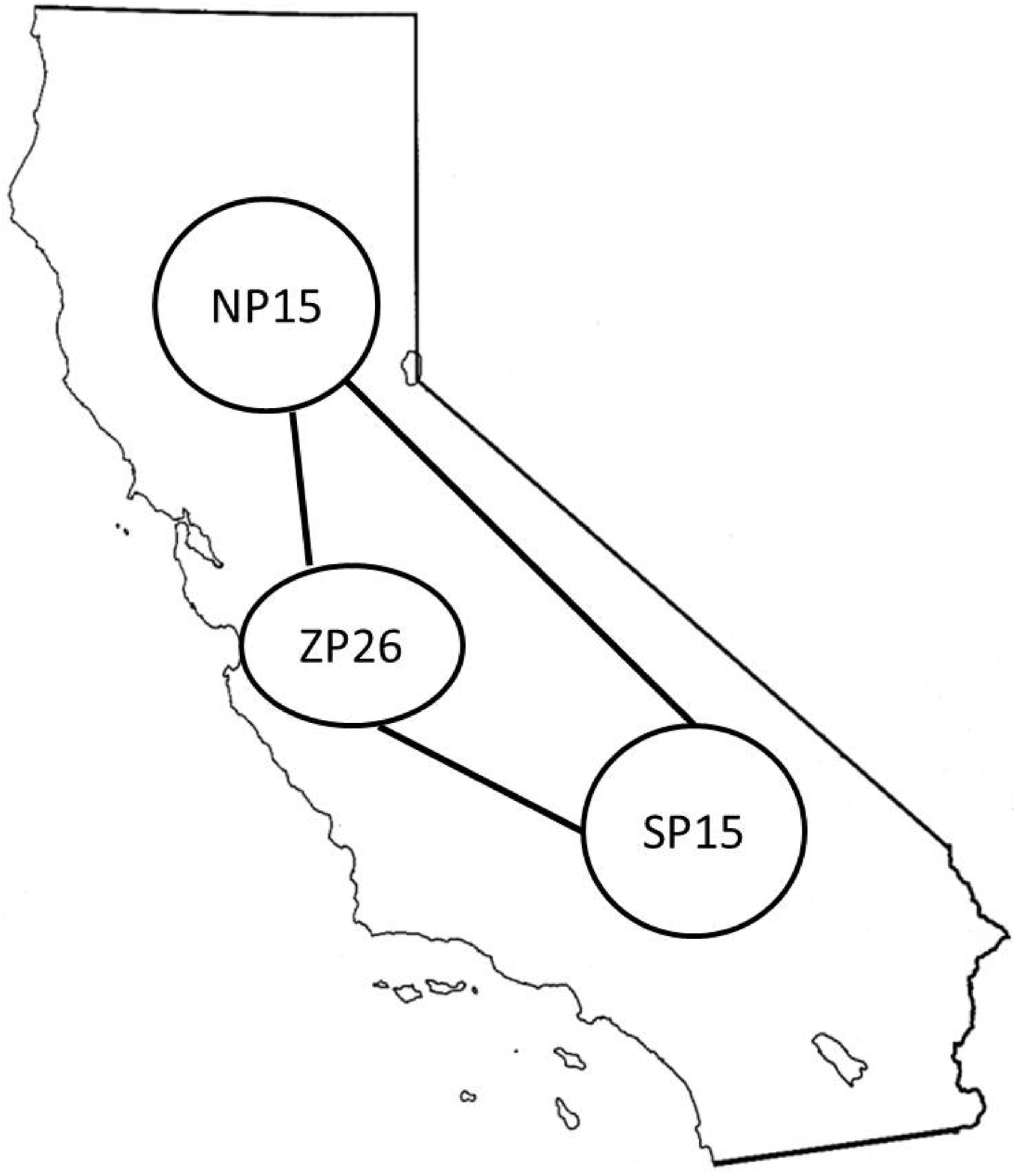}%
\caption{A zonal map of the California transmission network under CAISO control. The subnetwork within a zone are uncongested under normal operation. The tie lines to other WECC areas are not shown.}%
\label{fig:CA_zone}%
\end{figure}
Most networks consists of a large number of buses and lines, but under normal operating conditions, only very few lines are congested. For example, the California transmission network can be thought as divided into three zones connected by major transmission lines shown in Figure~\ref{fig:CA_zone} and the flows within a zone are unrestricted \cite{CAISO_report}. The zonal grouping in CAISO was designed utilizing the idea of collapsing together buses connected by uncongested transmission lines in a deterministic OPF.  We formalize and extend this intuitive concept for ND-RLD by showing that a general network with a single congested link reduces to  a two bus problem under mild to moderate uncertainty. More concretely, assume the line from bus $1$ to bus $2$ is congested, then 
\begin{thm} \label{thm:general}
Given a generic power network. Let $\ov{g}$ and $\ov{f}$ be the nominal generation and nominal line flows obtained by solving the nominal OPF (equation \eqref{eq:NDA-OPF}). Under the small-$\sigma$ assumption, suppose that $\ov{f}_{12}=c_{12}$ is the only congestion in the network, then the following holds:
\begin{enumerate}
	\item There are at most two nodes with positive generation. That is, $\ov{g}_i >0$ for at most two $i$. Furthermore, if $\ov{g}_1 \leq 0$, then only one other bus has positive generation.
	\item The risk limiting dispatch (equation \eqref{eqn:DA-SPF}) takes the form
	\begin{equation*}
	g^* = (\ov{\g}+\bm{\Del})^+,
	\end{equation*}
	where $\Del_i \neq 0$ only if $\ov{g}_i > 0$.
	\item If $\be_i=\be_k=\be$ for all $i,k$, then optimization problem reduces to an equivalent problem over a congested two node network with congestion from bus $1'$ to $2'$ with correlated errors. Let $k \neq 1$ be the bus with positive generation. Then the first stage costs are $\al_1'=\al_1$ and $\al_2'=(\frac{\al_k}{\gamma_k}-\gamma_k \al_1)$ and the errors are given by
	\begin{equation}
	e_1'=e_1+\sum_{i=3}^n \gamma_i e_i \;\;\;\;\;\; e_2'=e_2+\sum_{i=3}^n (1-\gamma_i) e_i,
	\end{equation}
	where $\gamma_i\in [0,1]$ are determined by the topology of the network and can be calculated by \eqref{eqn:g_flow} and $\eqref{eqn:gamma}$. 
\end{enumerate}
\end{thm}

{Point 1) in Theorem \ref{thm:general} seems strange since it is highly unlikely that only two generators would be generating in a power network. This result is comes from the assumption that the prices are linear in the power generated, which is used here to simplify the presentation. In practice, cost functions are piecewise linear or quadratic. If piecewise linear cost functions are used, then Theorem \ref{thm:general} 1) is modified to stating that there are at most two generators operating at their marginal cost \cite{Kirschen04}; if quadratic (or other convex continuous increasing) cost functions are used, Theorem \ref{thm:general} is modified to stating that there are at most two different marginal costs among the generators. The details of the derivation is given in the Appendix. The overall message of Theorem \ref{thm:general} remains unchanged in each case: in a network with one congested link, the risk limiting dispatch can be calculated by considering a two-bus network obtained from the original n-bus network.}   

{The proof of this theorem is somewhat technical and is given in the appendix.  The theorem states that the network can be collapsed into a single bus or a two bus network, utilizing an appropriate averaging of the net demands. To understand how to calculate the bus averaging weights $\gamma_i$, it is convenient to simplify \eqref{eqn:DA-SPF} (with cost $\bm{\beta}$) by considering {\it fundamental flows} \cite{Bondy08}.  Pick one spanning tree in the network. This spanning tree has $n$ nodes and $n-1$ edges. The flows on these $n-1$ edges is called a fundamental flow, denoted by $\tl{\bd f} \in \R^{n-1}$.  These flows are fundamental in the sense that any flows, $\bd f$  in the network can be written in the form $ \bd f=\bd R \tl{\bd f}$, where $\bd R \in \R^{m \times n-1}$ is a constant matrix only depending on the chosen spanning tree.

The constraint \eqref{eqn:Kirchoff} can be eliminated and \eqref{eqn:RT_X} reduces to:
\begin{subequations}\label{eqn:J_full_A}
\begin{align}
J^*(\bm{\be},\bd{x}) =\min & \bm{\be}^T (\g^R)^+ \\
 \mbox{subject to } & \g^R-\bd{x} - \bd{A} \bd{\tl{f}} =0 \label{eqn:nabla_A} \\
 & |\bd{R} \bd{f}| \leq \bd{c} \label{eqn:capacity_A}.
\end{align}
\end{subequations}}
  Let $\bd{a}_i^T$ be the $i$th row of $\bd{A}$ for $i=1,\dots,n$.  For each node $i=3,\dots,n$ in the network, let $\bd{\tl{f}}^{(i)}$ be set of fundamental flows that solve the following set of equations
\begin{subequations} \label{eqn:g_flow}
\begin{align}
f_1^{(i)} &= 0 \\
\bd{a}_i^T \bd{\tl{f}}^{(i)} &= -1 \\
\bd{a}_k^T \bd{\tl{f}}^{(i)} &=0, \; k \neq i, k \geq 3.
\end{align}
\end{subequations}
In matrix form, $\bd{\tl{f}}^{(i)}$ solves
\begin{equation*}
\bma
1 \; 0 \; 0 \; \cdots \; 0 \\
\bd{A}_2
\ebma \bd{\tl{f}}^{(i)} = \tl{\bd{A}} \bd{\tl{f}}^{(i)}= - \bd{h}_{i-1},
\end{equation*}
where $\bd{A}_2$ is the $(n-2) \times (n-1)$ matrix obtained by removing the first two rows of $\bd{A}$ and $\bd{h}_{i-1}$ is a vector with entry $i-1$ being 1 and all other entries $0$ . Inverting gives $\bd{\tl{f}}^{(i)}= -\tl{\bd{A}}^{-1} \bd{h}_{i-1}$ and
\begin{equation} \label{eqn:gamma}
\gamma_i= \bd{a}_1^T \bd{\tl{f}}^{(i)}.
\end{equation}

Next we apply Theorem \ref{thm:general}  to a three bus single cycle network with equal admittance on each line. Let the prediction $\hd$ be such that the line from bus 1 to bus 2 is congested. That is, $\ov{f}_{12}=c_{12}$ in the nominal problem. There are four possible congestion patterns\footnote{Other patterns are possible, but  occur for a set of $\hat{d}$ that is of measure zero} as listed in Figure \ref{fig:nom}. Bus $i$ is labeled by the sign of $\ov{g}_i$.
\begin{figure}[ht]
\centering
\subfigure[$ $]{
\includegraphics[scale=0.7]{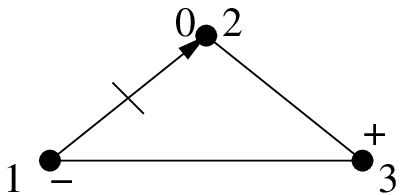}
\label{fig:nom1}}%
\subfigure[$ $]{
\includegraphics[scale=0.7]{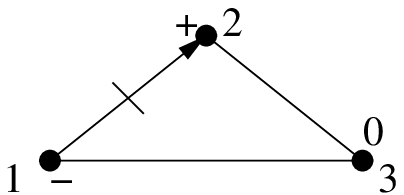}
\label{fig:nom2}}
\subfigure[$ $]{
\includegraphics[scale=0.7]{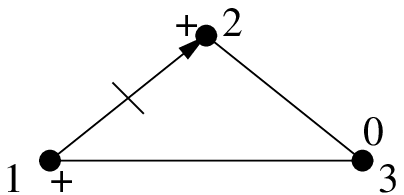}
\label{fig:nom4}}
\subfigure[$ $]{
\includegraphics[scale=0.6]{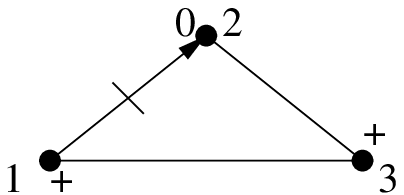}\label{fig:nom3}}
\caption{Possible sign patterns of $\ov{g}$ when a single line is congested.}%
\label{fig:nom}%
\end{figure}
Figure \ref{fig:equi_3} shows the equivalent two bus networks for each of the networks in Fig. \ref{fig:nom} after applying Theorem \ref{thm:general}. The networks in Fig. \ref{fig:equi_3} are labeled by the first stage costs, the sign patterns and the forecasted errors at each of the nodes. Let $\bD'$ be the solution to the two bus networks in Fig. \ref{fig:equi_3}. Then the controls $\bD$ for the original problem are given in each of the networks in Fig. \ref{fig:equi_3}.
\begin{figure}[h!]
\centering
\subfigure[$\Del_3=\Del_2'$]{
\psfrag{a1}{$\al_1$}
\psfrag{a2}{$\al_3$}
\psfrag{1}{$-$}
\psfrag{2}{$+$}
\psfrag{D1}{$e_1+\frac12 e_3$}
\psfrag{D2}{$e_2+\frac12 e_3$}
\includegraphics[scale=0.7]{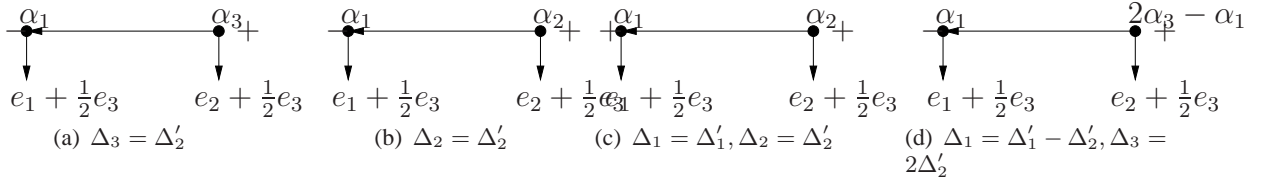}} \hspace{0.5cm}
\subfigure[$\Del_2=\Del_2'$]{
\psfrag{a1}{$\al_1$}
\psfrag{a2}{$\al_2$}
\psfrag{1}{$-$}
\psfrag{2}{$+$}
\psfrag{D1}{$e_1+\frac12 e_3$}
\psfrag{D2}{$e_2+\frac12 e_3$}
\includegraphics[scale=0.7]{equi_3.eps}}
\subfigure[$\Del_1=\Del_1',\Del_2=\Del_2'$]{
\psfrag{a1}{$\al_1$}
\psfrag{a2}{$\al_2$}
\psfrag{1}{$+$}
\psfrag{2}{$+$}
\psfrag{D1}{$e_1+\frac12 e_3$}
\psfrag{D2}{$e_2+\frac12 e_3$}
\includegraphics[scale=0.7]{equi_3.eps}} \hspace{0.5cm}
\subfigure[$\Del_1=\Del_1'-\Del_2', \Del_3=2 \Del_2'$]{
\psfrag{a1}{$\al_1$}
\psfrag{a2}{$2\al_3-\al_1$}
\psfrag{1}{$-$}
\psfrag{2}{$+$}
\psfrag{D1}{$e_1+\frac12 e_3$}
\psfrag{D2}{$e_2+\frac12 e_3$}
\includegraphics[scale=0.7]{equi_3.eps}}
\caption{The equivalent perturbed networks for the networks in Fig. \ref{fig:nom} respectively. The left bus is $1'$ and the right bus is $2'$. The back flow is only allowed form $2'$ to $1'$.}%
\label{fig:equi_3}%
\end{figure}

Note the result in this section can be extended to the case of a network with multiple congested lines. Namely, given a network with $K$
congested lines, it can be reduced to an equivalent network with $K + 1$ buses \cite{Zhang13cdc}. The methods for multiple congested
lines are the same for a single congested line, although the mathematical details are more cumbersome to handle.

\section{Simulation Results} \label{sec:simulations}

This section explores various numerical examples using the IEEE 9-bus benchmark network. In particular we compare the performance of ND-RLD with utilizing the standard $3-\sigma$ rule. We also compute the price of uncertainty numerically and compare it to the theoretical prediction. 
\vspace{-0.5cm}
\subsection{Uncongested Network}
Many practical networks  have line capacities that are much larger than the typical power flows. For these networks, they are well approximated by a single bus network. For example, consider the IEEE 9-bus network in Figure~\ref{fig:9bus}.
\begin{figure}[ht]
\centering
\includegraphics[scale=0.2]{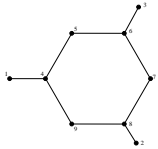}%
\caption{IEEE 9-bus benchmark network. Bus 1,2,3 are generators and the rest of the buses are loads.}%
\label{fig:9bus}%
\vspace{-0.1cm}
\end{figure}
The nominal generation and demands from the data included with this benchmark network \cite{IEEEbenchmark,Zimmerman09} is shown in Table~\ref{tab:branch}. Note that line flows  are significantly smaller than transmission line capacities. Therefore, under moderately high penetration, the network can be thought as a network operating without capacity constraints.

Up to this point we have used the DC power flow model, while in reality power flow is AC. It is known that for transmission networks, due to the low $R/X$ ratios of the transmission lines, DC and AC power flows yields similar answers. This is confirmed in our simulations where the difference in performance of using the risk limiting dispatch under DC and AC power flow models is minimal. Therefore it is sufficient to use the simpler DC flow model to obtain the dispatch.

\begin{table}[ht]
\centering
\begin{tiny}
\begin{tabular}{c|c|c|c|c|c|c|c|c|c}
Bus & 1 & 2 & 3 & 4 & 5 & 6 & 7 & 8 & 9 \\
\hline
DC Flow & 86.6 & 134.4 & 94.1& 0 &-90 & 0 & -100 & 0 & -125 \\
AC Flow & 89.8 & 134.3 & 94.2 & 0 & -90 & 0 & -100 & 0 & -125
\end{tabular}
\caption{All units are MW. Negative numbers are the demands at buses 5, 7, and 9. The generations needed at buses 1, 2, and 3 to meet these demands under both DC flow and AC flow are shown.}
\label{tab:bus}
\end{tiny}
\vspace{-0.5cm}
\end{table}
\begin{table}[ht]
\centering
\begin{tiny}
\begin{tabular}{c|c|c|c|c|c|c|c|c|c}
From bus & 1 & 4 & 5 & 3 & 6 & 7 & 8 & 8 & 9 \\
\hline
To bus & 4 & 5 & 6 & 6 & 7 & 8 & 2 & 9 & 4 \\
\hline
DC Flow & 86.6 & 33.7 & -56.3 & 94.1 & 37.8 & -62.2 & - 134.4 & 72.2 & -52.8 \\
\hline
AC Flow & 89.8 & 35.2 & -55.0 & 94.2 & 38.2 & -61.9 & -134.3 & 72.11 & -54.3 \\
\hline
Capacity & 250 & 250 & 150 & 300 & 150 & 250 & 250 & 250 & 250
\end{tabular}
\caption{All units are MW. Both DC and AC power flows on each line of the network is shown. Capacities are the long term emergency rating of the line. The network is uncongested.}
\label{tab:branch}
\end{tiny}
\vspace{-0.6cm}
\end{table}

To analyze the performance of the risk limiting dispatch derived in Section \ref{sec:single}, we compare it to two other dispatches. The first one is the currently used $3-\sigma$ dispatch, and the second one is the oracle dispatch where the actual realization of the wind is known at stage 1. We assume that all the generating buses have a first stage cost\footnote{The nominal generations are determined by an OPF problem, and every generator with non-zero generation has the same marginal cost. This can be thought as $\al$.} $\al=1$ and all buses have the same second stage cost $\be$. For simplicity, the prediction errors are generated as i.i.d. zero mean Gaussian random variables with variance $\sigma^2$. The predictions $\hd$ is taken to the nominal demands in Tab. \ref{tab:bus}.

The risk limiting dispatch is derived by viewing the network as a single bus. For actual operation, the amount of reserves to put at each buses in the network need to be determined. Here we spread the reserves equally among the three generating buses(buses 1,2 and 3). From \eqref{eqn:rld} and the fact that the prediction errors are independent, the single bus risk limiting dispatch is $\sum_{i=1}^9 \hat{d}_i + \Del$ where $\Del=\sqrt{9}\sigma Q^{-1}(\frac{\al}{\be})$. The network risk limiting dispatch is given by
\begin{align*}
\g_{\rld} = \ov{g}+\bD = & \bma 86.6 & 134.4 & 94.1 & 0 & \dots & 0 \ebma^T \\
&+ 3 \sigma Q^{-1}(\frac{\al}{\be}) \bma \frac13 & \frac13 & \frac13 & 0 & \dots & 0 \ebma^T.
\end{align*}
The $3-\sigma$ control purchases a reserve of 3 times the standard deviation for each bus in the network, or $3\cdot 9 \cdot \sigma$. Again we spread out the $3-\sigma$ dispatch over the three generating nodes as
\begin{align*}
\g_{\rld} = \ov{g}+\bD = & \bma 86.6 & 134.4 & 94.1 & 0 & \dots & 0 \ebma^T \\
&+ 9 \sigma Q^{-1}(\frac{\al}{\be}) \bma 1 & 1 & 1 & 0 & \dots & 0 \ebma^T.
\end{align*}
We simulate the cost for both the DC and AC power flows.

Figure \ref{fig:total15_iso} plots the total cost of the three dispatches for $\be=1.5 \al$. As we can see the risk limiting dispatch performs much better than the $3-\sigma$ dispatch. There are two reasons why the $3-\sigma$ dispatch or rules like it perform badly. The first is that the $3-\sigma$ rules is too conservative since it does not take into account the actual cost of the second stage; the second reason is that the $3-\sigma$ dispatch ignores the potential benefit of averaging between the prediction errors by treating the different buses as isolated nodes. In contrast, risk limiting dispatch takes these two points into consideration.
\begin{figure}[ht]
\centering
\includegraphics[scale=0.26]{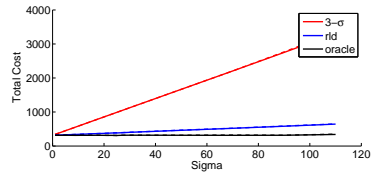}%
\caption{Total costs for $\be=1.5 \al$ as a function of $\sigma$. The red, blue and black lines are the total cost for the $3-\sigma$, risk limiting, and the oracle dispatches respectively. The solid lines are the costs under DC flow while the dotted lines are for AC flows.}%
\label{fig:total15_iso}%
\vspace{-0.3cm}
\end{figure}
Figure \ref{fig:total15_rld} is a zoomed in version of Fig. \ref{fig:total15_iso} by plotting the total cost only for the risk limiting dispatch and the oracle dispatch. The cost for the oracle dispatch is constant at 315  up until $\sigma=80$. This is expected since the predicted total demand is $315$ MW, and the prediction errors are zero mean, so the errors averages out. At higher $\sigma$, the capacities in the network become binding and the cost goes up since not all errors can be averaged. The cost for the risk limiting dispatch is essentially linear for all $\sigma$'s. Furthermore, the slope of the cost is (very close to) the price of uncertainty calculated in the earlier sections. 

 A lower bound for the minimum total cost is the total cost of applying the risk limiting dispatch to a network with infinite capacities, since an infinite capacity network has lower cost than a finite capacity one and the risk limiting dispatch is optimal for the former. From Figure \ref{fig:total15_rld}, this lower bound is almost met. Thus the risk limiting dispatch is close to optimal and our assumption of viewing an uncongested network as a single bus network is valid.
\begin{figure}[ht]
\centering
\includegraphics[scale=0.26]{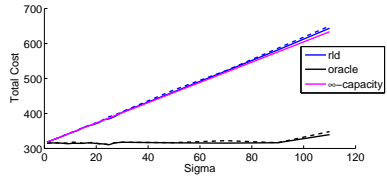}%
\caption{Total costs for $\be=1.5 \al$ as a function of $\sigma$. The blue and black lines are the total cost for the risk limiting and the oracle dispatches respectively. The purple line is the cost of the RLD when applied to an infinite capacity network, which is a lower bound for the minimum cost of the finite capacity network. The slopes of the blue and the purple lines represent the price of uncertainties.}%
\label{fig:total15_rld}%
\end{figure}
The slopes of the lines gives the price of uncertainties. As expected, the price of uncertainty for the oracle dispatch is $0$ since the actual realization is known at the first stage. The price of uncertainty of the risk limiting dispatch closely matches that of the single bus price of uncertainty, while the $3-\sigma$ price is much higher.
\subsection{Congested Network}
To construct a congested network, the network in Fig. \ref{fig:9bus} is modified by increasing the nominal load at bus 5 to $150$ MW and reducing the capacity of the line connecting bus 5 and 6 to $75$ MW. Then the line from bus 6 to bus 5 is congested. There are two different first stage costs $\al_1$ and $\al_2$ and these are given by the marginal costs of the generators. Let $\al=\frac12 (\al_1+\al_2)$ and we normalize all cost by $\al$.
Figure \ref{fig:total_cong_15} plots the total cost of the three dispatches for $\be=1.5 \al$. Again, we see the risk limiting dispatch performs much better than the $3-\sigma$ dispatch.
\begin{figure}[ht]
\centering
\includegraphics[scale=0.26]{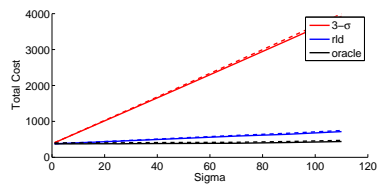}%
\caption{Total costs for $\be=1.5 \al$ as a function of $\sigma$. The red, blue and black lines are the total cost for the $3-\sigma$, risk limiting, and the oracle dispatches respectively. The solid lines are the costs under DC flow while the dotted lines are for AC flows. The purple line is the cost of the rld when applied to a network where only one line has finite capacity, namely the line congested under the nominal flows. This is a lower bound for the minimum cost of the finite capacity network. The slopes of the blue and the purple lines represent the price of uncertainties.}%
\label{fig:total_cong_15}%
\vspace{-0.5cm}
\end{figure}
Figure \ref{fig:total_cong_rld_15} is a zoomed in version of Fig. \ref{fig:total_cong_15} with the total cost only for the risk limiting dispatch and the oracle dispatch. As expected, the cost of the oracle dispatch is constant over a wide range of $\sigma$'s. The cost of the risk limiting dispatch is linear and very close to its lower bound. The lower bound is obtained by applying the risk limiting dispatch to a network with only one finite capacity line, namely the line congested under the nominal flows. Figure \ref{fig:total_cong_rld_15} shows that modeling a network with one congested line as a two bus network is very accurate.
\begin{figure}[ht]
\centering
\includegraphics[scale=0.26]{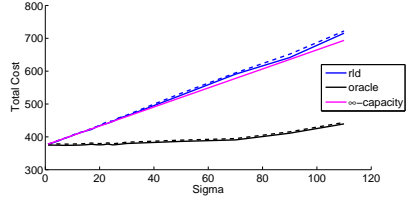}%
\caption{Total costs for $\be=1.5 \al$ as a function of $\sigma$. The blue and black lines are the total cost for the risk limiting and the oracle dispatches respectively. The purple line is the cost of the RLD when applied to a network where only one line has finite capacity, namely the line congested under the nominal flows. This is a lower bound for the minimum cost of the finite capacity network. The slopes of the blue and the purple lines represent the price of uncertainties.}%
\label{fig:total_cong_rld_15}%
\vspace{-0.5cm}
\end{figure}

Figure \ref{fig:total_cong_compare} shows the difference in cost of assuming there is no congestion in the network and the correct dispatch solution taking the congestion into account. The former calculation ignores the congestion information in the network. As expected, the later dispatch performs better since it takes into account the congestion in the network. 
\begin{figure}[ht]
\centering
\includegraphics[scale=0.26]{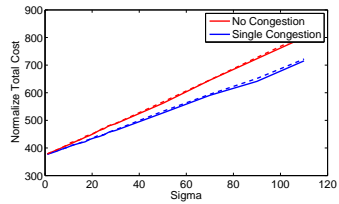}%
\caption{Total costs for $\be=1.5 \al$ as a function of $\sigma$. The blue line is the total cost for the risk limiting dispatch developed in Section \ref{sec:Networks}. The red line is the total cost if the risk limiting dispatch derived for the congested network in Section \ref{sec:single} is used. }%
\label{fig:total_cong_compare}%
\vspace{-0.5cm}
\end{figure}
\section{Conclusion} \label{sec:con}
In this paper we addressed the solution of a two-stage stochastic dispatch for system operators. We showed that a simple control exists under mild to moderate uncertainty about future realizations of net demand. The control is composed of two parts, one which is the certainty equivalent control rule, and another that is a deviation that hedges against the uncertainty by appropriately taking into account costs and recourse opportunities. Moreover, by incorporating the fact that only a small number of transmission lines that congest at any given hour, the optimal dispatch can be calculated analytically.   The price of uncertainty is a tool to measure the performance of distinct dispatch procedures. We show that under mild assumptions on forecast errors, the proposed dispatch achieves the cost bound given by the price of uncertainty. The proposed procedure also performs rather well in a full AC network.
\bibliographystyle{IEEETran}
\bibliography{mybib}
\vspace{-0.5cm}
\appendix
\subsection{Proof of Theorem \ref{thm:price_single}} \label{app:price_single}
Mathematically, the \emph{small-$\sigma$} assumption means that we are operating in the scaling regime where $\frac{1}{\sigma_e} \hat{d} \rightarrow \infty$.
Under this assumption,
\begin{subequations}
\begin{align*}
C(\hat{d})&= \lim_{\frac{1}{\sigma_e} \hat{d} \rightarrow \infty} \{\min_g \; \al g + \be \E[(d-g)^+|\hat{d}, d>0] \nn \\
 &- \al \E[d^+|\hat{d},d>0] \} \\
          &\stackrel{(a)}{=} \lim_{\frac{1}{\sigma_e} \hat{d} \rightarrow \infty} \{\min_g \; \al g + \be \E[(d-g)^+|\hat{d},d>0]  \nn \\
          &- \al\E[d|\hat{d},d>0]\} \\
          &= \lim_{\frac{1}{\sigma_e} \hat{d} \rightarrow \infty}\{\min_g\; \al g + \be\E[(\hat{d}+e-g)^+|\hat{d},d>0] \\
          &  - \al \E[d|\hat{d},d>0]\} \\
          &\stackrel{(b)}{=} \lim_{\frac{1}{\sigma_e} \hat{d} \rightarrow \infty}\{\min_\Del \; \al(\hat{d}+\Del) + \be\E[(e-\Del)^+|\hat{d},d>0] \\
          & -\al \E[\hat{d}+e|\hat{d},d>0]\} \\
          & =\lim_{\frac{1}{\sigma_e} \hat{d} \rightarrow \infty}\{ \min_\Del \; \al \Del + \be\E[(e-\Del)^+|\hat{d},d>0] \\
          & -\alpha \E[e|\hat{d},d>0]\}\\
          &\stackrel{(c)}{=} \sigma_e \lim_{\frac{1}{\sigma_e} \hat{d} \rightarrow \infty}\{ \min_{\Del'} \; \al \Del' + \be  \E[(z-\Del')^+|\hat{d},d>0]\} \label{eqn:Del'} \\
          &\stackrel{(d)}{=} \sigma_e p,
\end{align*}
\end{subequations}
where $(a)$ follows from the assumption $d>0$, $(b)$ follows from setting $g=\hat{d}+\Del$, $(c)$ follows from changes from variables where $\Del'=\Del/\sigma_e$ and $z=e/\sigma_e$ and the mean of $e$ remains $0$ in the limit and $(d)$ is follows the calculation below. 

From first order conditions, the optimal solution $\Delta'^*$ solves
\begin{subequations}
\begin{align*}
\alpha & =\beta \lim_{\frac{1}{\sigma_e} \hat{d} \rightarrow \infty} \Pr(z>\Delta'^*|\hat{d},d>0)\\
&= \beta \lim_{\frac{1}{\sigma_e} \hat{d} \rightarrow \infty} \int_{\min(\Delta'^*,-\frac{1}{\sigma_e} \hat{d})}^{\infty} \phi(x) dx=\beta Q(\Delta'^*). 
\end{align*}
\end{subequations}
Therefore,$\Delta'^*=Q^{-1}(\frac{\alpha}{\beta})$. The price of uncertainty $p$ can be calculated as
\begin{subequations}
\begin{align*}
p &= \lim_{\frac{1}{\sigma_e} \hat{d} \rightarrow \infty}\{\al Q^{-1}(\frac{\al}{\be})+ \be \E[(z-Q^{-1}(\frac{\al}{\be}))^+|\hat{d},d>0] \} \\
  &= \al Q^{-1}(\frac{\al}{\be}) \\
  &+ \be \lim_{\frac{1}{\sigma_e} \hat{d} \rightarrow \infty}\{ \int_{\min(Q^{-1}(\frac{\alpha}{\beta},-\frac{1}{\sigma_e} \hat{d})}^{\infty} (z-Q^{-1}(\frac{\al}{\be}))\phi(z) dz\} \\
  &= \al Q^{-1}(\frac{\al}{\be})+\be (-\frac{\al}{\be}Q^{-1}(\frac{\al}{\be})+ \phi(Q^{-1}(\frac{\al}{\be}))) \\
  &= \be \phi(Q^{-1}(\frac{\al}{\be})).
\end{align*}
\end{subequations} 

\subsection{Proof of Theorem \ref{thm:general}}
By assumption only the flow from 1 to 2 is congested, \eqref{eqn:J_full_A} can be replaced by an equivalent problem by choosing $f_{12}$ as a fundamental flow and including only the constraint on it. Without loss of generality, let $\tl{f}_1=f_{12}$.
\begin{subequations}\label{eqn:DCOPF12}
\begin{align}
J^*(\bm{\be},\bd{d}-\g) =\min & \bm{\be}^T (\g^R)^+ \\
 \mbox{subject to } & \g^R-(\bd{d}-\g) - \bd{A} \bd{\tl{f}} =0  \\
 &  \tl{f}_1 \leq C_{12}.
\end{align}
\end{subequations}
Writing \eqref{eqn:DCOPF12} as a linear program gives
\begin{subequations} \label{eqn:DCOPFlin}
\begin{align}
\mbox{minimize } & \bm{\al}^T \y \\
\mbox{subject to } & \y - \bd{A} \bd{\tl{f}} -\hd \geq 0 \\
& \tl{f}_1 \leq C_{12}\\
& \y \geq 0.
\end{align}
\end{subequations}
The Lagrangian is
\begin{equation*}
\mc{L}= \bm{\al}^T \y - \bm{\la}^T(\y -\bd{A} \bd{\tl{f}}-\hd) + \mu (\tl{f}_1-C_{12})-\bd{\nu}^T \y,
\end{equation*}
where $\bm{\la}$, $\mu$, and $\bd{\nu}$ are the Lagrangian multipliers. Differentiating with respect to $\y$ gives $\bm{\al} - \bm{\la} -\bm{\nu} =0.$

Since $\bm{\nu}$ are the Lagrangian multipliers associated with the constraint $\y \geq 0$, but complementary slackness $y_i >0$ only if $\nu_i=0$ or $\la_i=\al_i$. Equivalently, $\ov{g}_i >0$ only if $\nu_i=0$ or $\la_i=\al_i$. Differentiating with respect to $\bd{\tl{f}}$ gives
\begin{equation} \label{eqn:la_mu}
\bd{A}^T \bm{\la} + \mu \bd{h}_1 =0,
\end{equation}
where $\bd{h}_1=[1 \; 0 \; \cdots \; 0]^T$ is the first standard basis.
The dual is
\begin{subequations} \label{eqn:DCOPFdual}
\begin{align}
\mbox{maximize } & \bm{\la}^T \hd - \mu C_{12} \\
\mbox{subject to } & 0 \leq \bm{\la} \leq \bm{\al} \\
& \bd{A}^T \bm{\la} + \mu \bd{h}_1 =0\\
& \mu \geq 0.
\end{align}
\end{subequations}
At first glance it seems that the dimension of \eqref{eqn:DCOPFdual} is $n+1$. However since \eqref{eqn:la_mu} is $n-1$ equations involving $n+1$ unknowns, there are only $2$ independent variables. The next claim gives a precise characterization of the solution of \eqref{eqn:DCOPFdual}.
\begin{claim} \label{clm:1}
The solutions to \eqref{eqn:DCOPFdual} are in the forms of
\begin{equation*}
\bm{\la} = \bma 1 \\0 \\ \gamma_3 \\ \vdots \\ \gamma_n \ebma \la_1 + \bma 0 \\ 1 \\ 1-\gamma_3 \\ \vdots \\ 1-\gamma_n \ebma \la_2,
\end{equation*}
where $\gamma_i \in [0,1]$ for $i=3,\dots,n$ and $\mu= k (\la_2 -\la_1)$,
where $k$ is a positive constant depending on the graph structure. 
\end{claim}

Suppose the claim is true.
The first statement of Theorem \ref{thm:general} is that only two nodes are generating energy. From complementary slackness, $\ov{g}_i >0$ only if $\la_i = \al_i$. Since $\bm{\al}$ has only two degrees of freedom, for generic $\bm{\al}$, $\la_i=\al_i$ for at most two components. Therefore in general only two nodes would be generating energy.

The second statement is that only the nodes that generates power would be used to do the perturbation control. That is, $\Del_i \neq 0$ only if $\ov{g}_i >0$. The intuition is as follows: suppose $\ov{g}_i <0$, then under the small sigma assumption, $\ov{g}_i$ can be viewed as an infinite source of free energy, so no perturbation is needed; suppose $\ov{g}_i =0$, if a small unit of energy is purchased at node $i$, there is a cheaper option to purchase the unit of energy somewhere else (or $\ov{g}_i$ would have been positive), therefore $\Del_i=0$.

To show that the problem reduces to a two bus network if  all $\be$ are equal, we need to consider the second stage optimization problem. Now let $\bd{f}$ denote the set of perturbed flows. Since the line from 1 to 2 is congested in the nominal problem, $\tl{f}_1=f_{12}<0$. Let $x_i= \Del_i + (-\ov{g}_i)^+/\sigma-z_i$, where $\Del_i$ is the first stage control, $(-\ov{g}_i)$ is the left over energy, and $z_i$ is the normalized estimation error. The second stage optimization problem becomes
\begin{subequations} \label{eqn:per}
\begin{align}
J(\be,\ov{g})=\mbox{minimize } & \bm{\be}^T \y \\
\mbox{subject to } & \y - \bd{A} \bd{\tl{f}} +\x \geq 0 \\
& \tl{f}_1 \leq 0\\
& \y \geq 0.
\end{align}
\end{subequations}
This optimization problem has precisely the same form as \eqref{eqn:DCOPFlin}, with $C_{12}=0$. By Claim \ref{clm:1}, the dual of \eqref{eqn:per} is
\begin{subequations} \label{eqn:per_dual}
\begin{align}
\mbox{maximize } & - \la_1 \bm{\gamma}^T \x-\la_2 (\bd{1}-\bm{\gamma})^T \x   \\
\mbox{subject to } & 0 \leq \bm{\gamma} \la_1 + (\bd{1}-\bm{\gamma}) \la_2 \leq \bm{\be} \\
& \la_2-\la_1 \geq 0,
\end{align}
\end{subequations}
where $\bm{\gamma}=[ 1 \; 0 \; \gamma_3 \; \cdots \; \gamma_n]^T$, $a_i \in [0,1]$ and depends on the network topology for $i=3,\dots,n$. If $\be$'s are all the same (or $\be_1$, $\be_2$ are smaller than all other $\be$'s), the dual reduces to
\begin{subequations} \label{eqn:per_dual2}
\begin{align}
\mbox{maximize } & - \la_1 x_1'-\la_2 x_2'   \\
\mbox{subject to } & 0 \leq \la_1 \leq \be_1  \\
& 0 \leq \la_2 \leq \be_2 \\
& \la_2-\la_1 \geq 0,
\end{align}
\end{subequations}
where $x_1'= \bm{\gamma}^T \x$ and $x_2'=(\bd{1}-\bm{\gamma})^T \x$. This is exactly the dual of a two bus network with prediction errors $\bm{\gamma}^T \z$ and $(1-\bm{\gamma})^T \bd{z}$, leftover energy $\bm{\gamma}^T (-\ov{\g})^+$ and $(1-\bm{\gamma})^T (-\ov{\g})^+$,  and controls $\bm{\gamma}^T \bm{\Del}$ and $(1-\bm{\gamma}^T) \bm{\Del}$.

Let $\Del_1'=\bm{\gamma}^T \bm{\Del}$  and $\Del_2'=(1-\bm{\gamma})^T \bm{\Del}$. It can be shown that if there are two generating nodes then one of them can be taken to be node $1$. Suppose the other generating node is node $k$. To solve equilibrium equation \eqref{eqn:equi_2bus} for $\Del_1'$ and $\Del_2'$, the associated first stage costs are $ \al_1$ and $(\frac{\al_k}{\gamma_k}-\gamma_k \al_1)$ respectively.
\subsection{Proof of Claim \ref{clm:1}}
We prove Claim \ref{clm:1} be guessing the solution and verifying it is correct.  We show $\bm{\la}=[1 \; 0 \; \gamma_3 \; \cdots \; \gamma_n]^T$ where $\gamma_i$ is given by \eqref{eqn:gamma} solves \eqref{eqn:la_mu}. Expanding $\bd{A}^T \bm{\la}$ gives
\begin{align*}
\bd{A}^T \bm{\la} &= \sum_{i=1}^n \bd{a}_i \gamma_i \\
&= \bd{a}_1 + \sum_{i=3}^n  \bd{a}_i \gamma_i  \\
& = \bd{a}_1 + \sum_{i=3}^n \bd{a}_i (\bd{a}_1^T \bd{\tl{f}}^{(i)}) \\
& = \bd{a}_1 + \sum_{i=3}^n \bd{a}_i ( (\bd{\tl{f}}^{(i)})^T \bd{a}_1)  \\
& = \bd{a}_1 + \sum_{i=3}^n \bd{a}_i ( (-\tl{\bd{A}}^{-1} \bd{h}_{i-1})^T \bd{a}_1) \\
&= \bd{a}_1 - (\bd{A}_2^T (-\tl{\bd{A}}^{-1} \bd{h}_{i-1})^T)\bd{a}_1 \\
&= (\bd{I}-(\bd{A}_2^T (-\tl{\bd{A}}^{-1} \bd{h}_{i-1})^T)) \bd{a}_1 \\
& \stackrel{(a)}{=} (\bd{I} - \bma * & * & * & \cdots & * \\ 0 & 1 & 0 & \cdots & 0 \\
0 & 0 & 1 & \cdots & 0 \\
& & & \vdots & 0 \\
0 & 0 & 0 & \cdots & 1 \ebma) \bd{a}_1 \\
& = \bma * \\ 0 \\ \vdots \\ 0 \ebma,
\end{align*}
where $*$ denote a generic number and $(a)$ follows from observing that $(\tl{\bd{A}}^{-1} \bd{h}_{i-1})^T$ is the transpose of $(\tl{\bd{A}}^T)^{-1}$ without the first row and the following simple lemma
\begin{lem}
Let $\bd{X}$ be a $r \times r-1$ matrix and suppose the matrix $\bma \bd{h}_1 & \bd{X} \ebma$ is invertible with inverse $\bd{Y}$. Let $\bd{Y}_1$ be the matrix obtained from $\bd{Y}$ by removing the first row. Then
\begin{equation*}
\bd{X} \bd{Y}_1 = \bma * & * & * & \cdots & * \\ 0 & 1 & 0 & \cdots & 0 \\
0 & 0 & 1 & \cdots & 0 \\
& & & \vdots & 0 \\
0 & 0 & 0 & \cdots & 1 \ebma.
\end{equation*}
\end{lem}

We still need to show that $1 \geq 0\gamma_i \geq 0$. This can be done through graph theory, but it is simpler to recognize that the flows given by \eqref{eqn:g_flow} solves a DC OPF problem. For the $i$'th node, the optimization problem is to find the least generation need at nodes $1$ and $2$, satisfying a demand of $1$ unit at node $i$, $0$ demand at all other nodes, and no flow on the line between $1$ and $2$. Therefore $\gamma_i$ is the proportion of power that node $1$ produced, and $1-\gamma_i$ is the proportion of power that node $2$ produced.

The vector $[0 \; 1 \; 1-\gamma_3 \; \cdots \; 1-\gamma_n]$ is a solution to \eqref{eqn:la_mu} since $\bd{1}$ is in the null space of $\bd{A}^T$.

\subsection{Nonlinear Prices}
First consider the following nominal OPF problem where $c_i$'s are differentiable 
\begin{subequations} \label{eqn:pos}
\begin{align}
\mbox{minimize } & \sum_{i=1}^n q_i(g_i^+) \\
\mbox{subject to }& \g-\bd{A} \bd{\tl{f}} -\bd{\hat{d}} =0,
\end{align}
\end{subequations}
where $q_i(\cdot)$ is increasing and convex, $\bd{\tl{f}}$ are the fundamental flows and the line capacities are infinite. Rewriting \eqref{eqn:pos} as
\begin{subequations} \label{eqn:y}
\begin{align}
\mbox{minimize } & \sum_{i=1}^n q_i(y) \\
\mbox{subject to }& \y-\bd{A} \bd{\tl{f}} -\bd{\hat{d}} \geq 0 \\
& \y \geq 0.
\end{align}
\end{subequations}
The Lagrangian of \eqref{eqn:y} is
\begin{equation} \label{eqn:L1}
\mc{L}= \sum_{i=1}^n q_i(y) -\bm{\la}^T(\y-\bd{A} \bd{\tl{f}} -\bd{\hat{d}} \geq 0) - \bm{\nu}^T \y.
\end{equation}
At optimal, the primal and dual variables minimizes the Lagrangian, so $q_i'(y_i^*)- \la_i^*-\nu_i^*=0$
and
$\bd{A}^T \bm{\la}^* =0.$

The null space of $\bd{A}^T$ is spanned by the $\bd{1}$ vector, therefore $\bm{\la}^*= \la^* \bd{1}$ for some $\la^*$. By complementary slackness, node $i$ is generating if $\nu_i=0$. The nodes in the network can be divided into two groups, one group is all the generating node and the other is the non-generating nodes. For all the generating nodes, the marginal prices are all the same, that is, $q_i' (y_i^*)=q_k'(y_k^*)=\la^*$ if $y_i >0$ and $y_k >0$. From the perturbation control of view, the network consists of many nodes with the same first stage cost $\la^*$, therefore the perturbation control $\Del$ can be divided up in arbitrary fashion between those nodes.

Now suppose that the network has one congested link. The nominal problem becomes
\begin{subequations} \label{eqn:y2}
\begin{align}
\mbox{minimize } & \sum_{i=1}^n q_i(y) \\
\mbox{subject to }& \y-\bd{A} \bd{\tl{f}} -\bd{\hat{d}} \geq 0 \\
& \tl{f}_1 \leq C_{12} \\
& \y \geq 0.
\end{align}
\end{subequations}
The Lagrangian is
\begin{equation*}
\mc{L}= \sum_{i=1}^n q_i(y) -\bm{\la}^T(\y-\bd{A} \bd{\tl{f}} -\bd{\hat{d}} \geq 0) + \mu(\tl{f}_1-C_{12})- \bm{\nu}^T \y.
\end{equation*}
At optimal, $q_i'(y_i^*)- \la_i^*-\nu_i^*=0$
and 
$\bd{A}^T \bm{\la}^* + \mu^* \bd{h}_1 =0.$

From previous results, $\bm{\la}$ can be decomposed as $\bm{\la}= \bd{a} \la_1 + (\bd{1}-\bd{a}) \la_2$, where $a_1=1, a_2=0$ and $a_i\in [0,1]$ for $i=3,\dots,n$. The decision variable in the two stage optimization problem becomes $\Del_1'= \bd{a}^T \bm{\Del}$ and $\Del_2'=(\bd{1} -\bd{a})^T \bm{\Del}$, and the optimization problem is
\begin{equation*}
\min_{\Del_1',\Del_2'} \la_1^* \Del_1'+ \la_2^* \Del_2' + \E[J(\bm{\be},\Del_1',\Del_2')].
\end{equation*}
After solving for the optimal $\Del_1'$ and $\Del_2'$, we need to map it back to the original $\Del_i$'s. Again, if $y_i=0$, then we set $\Del_i=0$. For the rest of the nodes, $\Del_1'$ and $\Del_2'$ can be split up in any fashion as long as it is consistent.
\subsection{Piecewise Linear Cost Functions}
In practice piecewise linear cost functions (convex, increasing) are often used. We start again with the uncongested network. Since the cost functions are not differentiable, the Lagrangian in \eqref{eqn:L1} is not differentiable. We use subgradients instead. Given a function $f(\x): \R^n \rightarrow \R$, a vector $ \bd{s} \in \R^n$ is a subgradient of $f$ at $\x$ if $ f(\y)-f(\x) \geq \bd{s}^T (\y-\x)$ for all $\y$. The set of all subgradients at $\x$ is called the subdifferential and denoted $\partial f(\x)$. By the convexity, at optimality,
\begin{equation*}
s_i^*- \la_i^*-\nu_i^*=0
\end{equation*}
for some $s_i^* \in \partial q_i (y_i^*)$ and $\bd{A}^T \bm{\la}^* =0.$
We still have all $\la$'s are equal and the two stage problem only depends $\Del=\Del_1+\Del_2+\dots+\Del_n$. In contrast to the differentiable cost case, here the control is only done at one node, that is, $\Del=\Del_i$ for some node $i$ and $\Del_k=0$ if $k \neq i$. The reason is that all generating node except one would be operating at a corner point on their respective cost curves and it is optimal not to change those corner points. Only one generator would be operating at a point not at the corner, and all the control should be done at that point.
\begin{figure}[ht]
\centering
\psfrag{g}{$g$}
\psfrag{q}{$q$}
\subfigure[Corner Operating Point]{
\includegraphics[scale=0.3]{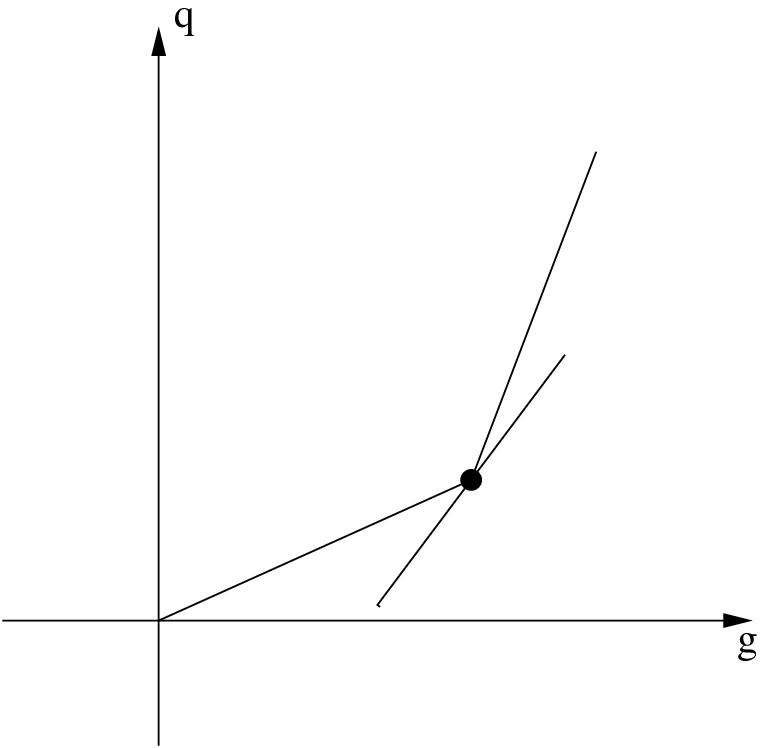}}
\subfigure[Non-corner Operating Point]{
\includegraphics[scale=0.3]{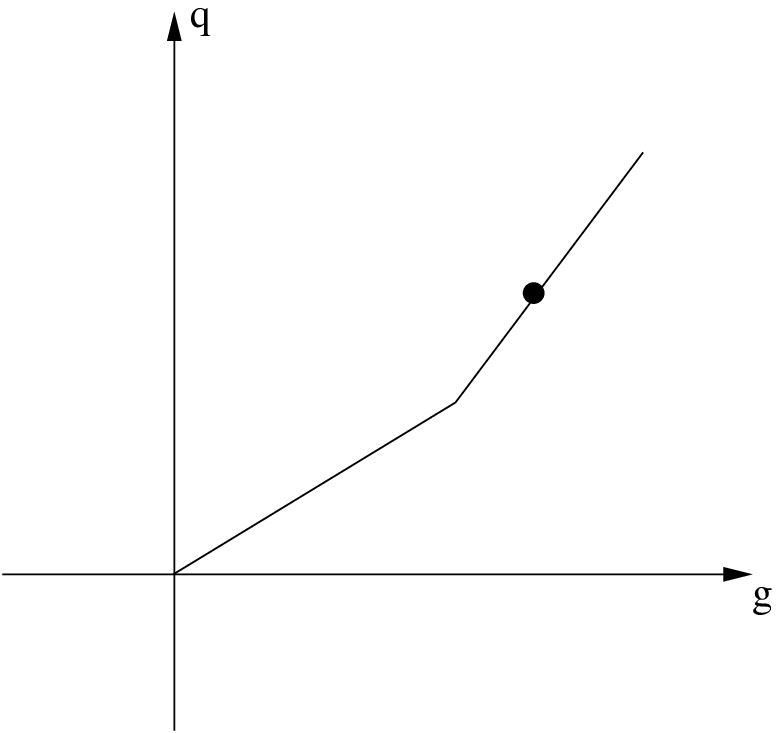}}%
\caption{All generating node except one will be operating at a corner point.}%
\label{fig:piece}%
\end{figure}
Similarly, for a network with one congested link, only two generators will be operating at a non-corner point, and the perturbation control is done at those two points.
\end{document}